\documentclass[smallextended,numbook,runningheads]{svjour3}     
\smartqed  
\usepackage{graphicx}
\usepackage{amsmath}
\usepackage{epstopdf} 
\usepackage{mathptmx}      
\usepackage{amssymb}
\usepackage{stmaryrd}
 \usepackage{mathtools}
\usepackage{tikz}
\usepackage{pgfplots}
\usepackage{pgfplotstable}
\usepackage{algorithm}
\usepackage{algorithmic}
 \usepackage{subcaption}
\usepackage{todonotes}

\DeclareMathOperator*{\argminB}{argmin}

\newcommand{\A}[0]{\mathcal{A}}
\newcommand{\Lam}[0]{\varLambda}
\newcommand{\Bf}[0]{\mathcal{B}}
\newcommand{\Lf}[0]{\mathcal{L}}
\newcommand{\Eps}[0]{\boldsymbol{\varepsilon}}
\newcommand{\Kappa}[0]{\boldsymbol{K}}
\newcommand{\ConstRel}[0]{\mathbb{C}\,}
\newcommand{\Mom}[0]{\boldsymbol{M}}
\newcommand{\Shear}[0]{\boldsymbol{Q}}
\newcommand{\Div}[0]{\mathrm{div}\,}

\newcommand{\Vn}[0]{V_{n}}
\newcommand{\mnn}[0]{M_{nn}}
\newcommand{\qn}[0]{Q_n}
\newcommand{\effs}[0]{V_n}

\newcommand{\mns}[0]{M_{ns}}
\newcommand{\msn}[0]{M_{sn}}

\newcommand{\Bh}[0]{\mathcal{B}_h}
\newcommand{\Lh}[0]{\mathcal{L}_h}

\newcommand{\osc}[0]{\mathrm{osc}}

\newcommand{\Ch}[0]{\mathcal{C}_h}

\newcommand{\Ehint}[0]{\mathcal{E}_h^I}

\newcommand{\sub}{R}

\newcommand{\vertiii}[1]{{\left\vert\kern-0.25ex\left\vert\kern-0.25ex\left\vert #1 
            \right\vert\kern-0.25ex\right\vert\kern-0.25ex\right\vert}}

\newtheorem{thm}{Theorem}
\newtheorem{lem}{Lemma}
\newtheorem{rem}{Remark}

\newtheorem{prob}{Problem}

\newtheorem{ass}{Assumption}

\numberwithin{equation}{section}

\journalname{BIT}
\begin{document}

\title{A stabilised finite element method for the plate obstacle problem\thanks{Funding from Tekes (Decision number 3305/31/2015), the Finnish Cultural Foundation,
the Portuguese Science Foundation (FCOMP-01-0124-FEDER-029408) and the Finnish Society of Science and Letters is greatly acknowledged.
}}


\author{Tom Gustafsson       \and
        Rolf Stenberg \and Juha Videman
}


\institute{Tom Gustafsson  \at
            Department of Mathematics and Systems Analysis,
Aalto University, P.O. Box 11100, 00076 Aalto, Finland\\
              \email{tom.gustafsson@aalto.fi}           
           \and
           Rolf Stenberg \at
             Department of Mathematics and Systems Analysis,
Aalto University, P.O. Box 11100, 00076 Aalto, Finland\\
    \email{rolf.stenberg@aalto.fi}   
    \and
    Juha Videman \at
    CAMGSD/Departamento de Matem\'atica, Instituto Superior T\'ecnico, Universidade   de Lisboa, Av. Rovisco Pais 1, 1049-001 Lisboa, Portugal \\
      \email{jvideman@math.tecnico.ulisboa.pt}
}

\date{Received: date / Accepted: date}

\maketitle

\begin{abstract}
We introduce a stabilised finite element formulation for the
Kirchhoff plate obstacle problem and derive both a priori and residual-based a posteriori error
estimates using conforming $C^1$-continuous finite elements. We implement the method as a Nitsche-type scheme and
give numerical evidence for its effectiveness in the case of an elastic and a rigid
obstacle.
\end{abstract}

\section{Introduction}
The goal of this paper is to introduce a stabilised finite element method 
for the obstacle problem of clamped Kirchhoff plates and perform an a priori and a posteriori error analysis based on conforming finite element approximation of the displacement field. 
To our knowledge, stabilised  $C^1$-continuous finite elements have not been previously analysed for fourth-order
obstacle problems. Moreover,  
only a few articles exist on the a posteriori error analysis of fourth-order
obstacle problems (cf. \cite{GudiPorwal2016,Brenneretal2017}) and none on
conforming $C^1$-continuous finite elements, most probably due to the limited
regularity of the underlying continuous problem. Here, we consider a stabilised  method
based on a
saddle point formulation which introduces the contact force as an
additional unknown (Lagrange multiplier).  We establish an a priori estimate with minimal regularity assumptions and derive residual-based a posteriori
error estimators. The Lagrange multiplier formulation
has the advantage of providing an approximation for the contact force
and the unknown contact domain. Moreover, it can easily be implemented as
a Nitsche-type method with only the primal displacement variable as an unknown in the
resulting linear system.

In a recent paper~\cite{GSV16}, we considered two families of finite
element methods for a second-order obstacle problem using a
Lagrange multiplier formulation for including the obstacle constraint. The first 
 was a family of mixed finite element methods  for which the discrete spaces need to
satisfy the Babu$\check{\rm s}$ka--Brezzi condition. This was achieved by  using ``bubble'' degrees of freedom. The second
 was a family of stabilised methods for which the stability is guaranteed, for all finite element space pairs, by
adding properly weighted residual terms to the discrete formulation.  In the analysis of the stabilised
formulation, we made use of recently developed tools  for the 
 Stokes problem \cite{SV15}.

In \cite{GSV16}, the analysis was focused on the membrane obstacle problem.
The approach followed  is, however,  quite general and should thus, up to some modifications, be extendable
to  other problems. In this paper, we consider conforming $C^{1}$-continuous elements for 
clamped Kirchhoff plates constrained by a rigid or elastic obstacle. This kind of
elements are rather complicated to work with and hence it does not seem reasonable to add
artificial bubble degrees of freedom, in particular since the bubbles should
belong to $H^{2} _{0}(K)$ at each element $K$. Therefore, we only address a
stabilised formulation.

Numerical approximation of fourth-order obstacle-type problems has been
previously studied in \cite{FusciardiScarpini80,Glowinskietal84,Scholtz87,Brenneretal2012a,Brenneretal2012b,Brenneretal2013,Brenneretal2017}.  In
\cite{FusciardiScarpini80,Glowinskietal84} the authors considered mixed finite
element methods and presented general convergence theorems without convergence
rates. In \cite{Scholtz87}, it was shown that using the penalty method and
piecewise quadratic elements, the method converges with the (suboptimal)  rate
of $h^{1/3}$ in the energy norm. Brenner {et al.}~\cite{Brenneretal2012a}
made a unified a priori error analysis for classical conforming and
non-conforming ($C^1$-continuous and $C^0$-continuous) finite element methods
(see, {e.g.},  \cite{CiarletBook}) as well as for  the $C^0$ interior
penalty methods and showed $\mathcal{O}(h)$ convergence rate for all methods in
convex domains, see also \cite{Brenneretal2012b,Brenneretal2013} for some
generalisations. The only existing a posteriori analyses on fourth-order
obstacle-type problems are due to Brenner {et al.}~\cite{Brenneretal2017}
and Gudi and Porwal~\cite{GudiPorwal2016}, both performed on the $C^0$ interior
penalty methods. In \cite{GudiPorwal2016}, the authors also derive a priori
error estimates with minimal regularity assumptions using the techniques
developed by Gudi in \cite{Gudi2010} much in the same spirit as we do here, see
also \cite{SV15,GSV16}.

All the above mentioned papers address the problem with a rigid obstacle. For
the plate bending problem with an elastic obstacle, we refer to
\cite{TosoneMaceri2003} for general convergence results in a mixed formulation
and to \cite{Hanetal2006} for optimal a priori estimates for conforming and
non-conforming methods in the primal formulation.

 The paper is organised as follows. In
Section~\ref{sec:cont}, we formulate the continuous problem  and
show its stability. In Section~\ref{sec:fem}, we define the stabilised finite
element method and establish a  discrete stability estimate as well as a priori
and a posteriori error estimates. In Section~\ref{sec:nitsche}, we derive
the corresponding Nitsche's method and discuss its implementation. Finally, in
Section~\ref{sec:numerical}, we report results of numerical computations on two example problems. 
In Sections 2 and 3, we will shorten (or omit) derivations that can be
inferred from our work on the Kirchhoff plate source problem~\cite{GSV18} and on the membrane obstacle
problem~\cite{GSV16}.


\section{The continuous problem}

\label{sec:cont}

Let us first recall the  Kirchhoff--Love theory for thin plates (see, {e.g.}, \cite{FS}).
We denote the infinitesimal strain tensor  as
\begin{equation}
    \Eps(\boldsymbol{v})=\frac12(\nabla \boldsymbol{v}+\nabla \boldsymbol{v}^T), \quad \forall \boldsymbol{v} \in \mathbb{R}^2,
\end{equation}
and consider the following isotropic linear elastic constitutive relationship, valid under plane stress conditions,
\begin{equation}
    \ConstRel \boldsymbol{A} = \frac{E}{1+\nu}\left( \boldsymbol{A} + \frac{\nu}{1-\nu} (\text{tr}\,\boldsymbol{A}) \boldsymbol{I}\right), \quad \forall \boldsymbol{A} \in \mathbb{R}^{2 \times 2},
\end{equation}
where $E$ and $\nu$ are the Young's modulus and the Poisson ratio. Letting $u$ stand for the deflection of the mid-surface of the plate and $d$ for the plate's thickness, the curvature $\Kappa$ and the bending moment $\Mom$ are defined as
\begin{equation}
    \Kappa(u) = -\Eps(\nabla u), \qquad \Mom(u) = \frac{d^3}{12} \ConstRel \Kappa(u).
\end{equation}
Assume that $\Omega \subset \mathbb{R}^2$ is a   polygonal domain occupied by  (the mid-surface of) the thin plate.
Since our interest lies in the obstacle problem, we will consider only clamped boundary conditions. The strain energy corresponding to a displacement $v$ of the plate
is 
$\frac{1}{2} a(v,v)$, with 
\begin{equation}
a(w,v) = \int_\Omega \Mom(w) : \Kappa(v) \,\mathrm{d}x.
\end{equation}
The displacement is constrained by an obstacle, denoted by $g$, which is allowed to be either rigid or elastic. The energy resulting from contact with an elastic obstacle can be written as
\begin{equation}
\frac{1}{2\epsilon} \int_{\Omega}(u-g)_-^2 \,\mathrm{d}x,
\end{equation}
where $\epsilon>0$ is the inverse of an appropriately scaled "spring constant" and 
\[
(u-g)_ {-} = \min\{u-g,0\}.
\]
The loading consists of a distributed load $f \in L^2(\Omega)$ with the potential energy
\begin{equation}
l(v)= \int_\Omega fv  \,\mathrm{d}x.
\end{equation}
The total energy thus reads as
\begin{equation}I(v)=\frac{1}{2} a(v,v) + \frac{1}{2\epsilon} \int_{\Omega}(v-g)_-^2 \,\mathrm{d}x- l(v).
\end{equation}

The space of kinematically admissible displacements is denoted by
$
V=H^2_0(\Omega).
$ The displacement function $u$ is thus obtained by minimising the energy, viz.
\begin{equation}
I(u) \leq I(v) \quad \forall v\in V,
\end{equation}
or by solving the weak formulation: Find $u\in V$ such that 
\begin{equation}  a(u,v) + \frac{1}{\epsilon} \big((u-g)_{-},v\big)  = l(v), \quad \forall v\in V,
\label{firstweak}
\end{equation}
where $(\cdot,\cdot) $ is the usual $L^2(\Omega)$ inner product.
The reaction force   between the obstacle and the plate  is given by 
\begin{equation}\label{lagdef}
\lambda=  -\frac{1}{\epsilon} (u-g)_{-}.
\end{equation}
In the limit $\epsilon\to 0$, the obstacle becomes rigid and the  problem reduces to that of  constrained minimisation
\begin{equation}\label{constprob}
u=\argminB_{v\in K} \Big[  \frac{1}{2} a(v,v)-l(v) \Big],
\end{equation}
with 
\begin{equation}\label{constraint}
K=\{ \, v\in H^2_0(\Omega) : v\geq g \ \mbox{in} \ \Omega\,\}.
\end{equation}
At the same time, the reaction force  $\lambda$ converges to the Lagrange multiplier associated with the constraint $v\geq g$.

The plate obstacle problem can be investigated based on the variational inequality formulation of problem \eqref{constprob} (see \cite{Scholtz87,Brenneretal2012a,Brenneretal2012b,Brenneretal2013}): Find $u\in K$ such that 
\begin{equation}
a(u, v-u)\geq l(v-u), \quad \forall v\in K.
\end{equation}
Here, we rewrite the problem using $\lambda$ as an independent
unknown to obtain a perturbed saddle point problem. From \eqref{lagdef} it
follows that the reaction force is non-negative, i.e. it belongs to  the set
\begin{equation}
    \Lam = \{ \mu \in Q : \langle v, \mu \rangle \geq 0~\forall v \in V  \text{~s.t.} \ v \geq 0 \text{~a.e.~in $\Omega$}\},
\end{equation}
where the function space for the Lagrange multiplier  is defined as
\begin{equation}
    Q = \begin{cases}
        L^2(\Omega), & \text{if $\epsilon>0$}, \\
        H^{-2}(\Omega), & \text{if $\epsilon=0$},
    \end{cases}
\end{equation}
and $\langle \cdot,\cdot \rangle : Q^\prime \times Q \rightarrow \mathbb{R}$ denotes the duality
pairing.
We denote by $\|\cdot\|_k$ the usual norm in the Hilbert space  $H^k(\Omega), k\in\mathbb{N}$, let $\|\cdot\|_0$ be the norm in $L^2(\Omega)$ and equip the space $ H^{-2}(\Omega)=[H^2_0(\Omega)]^\prime$  with the norm
\begin{equation}
    \|\xi\|_{-2} = \sup_{v \in V} \frac{\langle v, \xi \rangle}{\|v\|_2}.
    \label{negatnorm}
\end{equation}
Note that since the Lagrange multiplier in general belongs to $ H^{-2}(\Omega)$ in case of a rigid obstacle, the obstacle $g$ and the load $f$ could be such that the contact domain reduces to a point (or a finite number of points).

Under appropriate smoothness assumptions, the solution to the plate bending problem over a rigid obstacle is in $H^{3}_{\rm loc}(\Omega)\cap C^2(\Omega)$, in convex domains in $H^{3}(\Omega)$, cf.~\cite{Frehse71,CaffarelliFriedman79}, but it cannot belong to $H^4(\Omega)$. The exact solutions given in \cite{Brenneretal2013,Aleksanyan2016} seem to indicate that  the smoothness threshold is $C^{2,1/2}(\Omega)$ or $H^{7/2-\varepsilon}(\Omega), \varepsilon>0$. The solution to the clamped plate bending problem is more regular if the obstacle is elastic. In fact, assuming that the obstacle and the loading term are in $L^2(\Omega)$, the regularity of the solution is determined by the regularity of the source problem, cf.~\cite{Hanetal2006}. In particular, the solution belongs to $H^4(\Omega)$ if the interior angles of the domain $\Omega$ are smaller than $\approx126.284^\circ$, cf. \cite{BlumRannacher80}.

Formulation (\ref{firstweak}) can be written as: Find $u\in V$ and 
$\lambda \in \Lam$ such that 
\begin{align}
    a(u,v) - \langle v, \lambda \rangle &= l(v),  \quad \forall v\in V,
     \label{weakform1}  \\
    \langle u - g + \epsilon \lambda, \mu - \lambda \rangle &\geq 0, \quad \quad  \forall \mu \in \Lam.
   \label{weakform2} \end{align}
    
The stabilised finite element method exploits the strong form of the equations which we recall next. The static variables,   the moment tensor $\Mom$ and the  
shear force $\Shear$, satisfy the following equilibrium equations which, due to the loading and the Lagrange multiplier, have to be interpreted in the sense of distributions 
\begin{equation}
    \boldsymbol{\Div} \Mom(u) = \Shear(u), \qquad - \Div \Shear(u)-\lambda = l.
\end{equation}
A simple elimination leads to the equation
\begin{equation}
    \A(u) -\lambda = l,
\end{equation}
with the biharmonic operator $ \A(u)$ given by
\begin{equation}
    \A(u) := D \Delta^2 u,
\end{equation}
where  $D$ stands for the   bending stiffness defined through
\begin{equation}
    D = \frac{E d^3}{12(1-\nu^2)}.
\end{equation}

 The strong form of  problem (\ref{weakform1})-(\ref{weakform2}) is thus: Find $u$ and $\lambda$ such that
\begin{align}
    \left.
    \begin{aligned}
                  \mathcal{A}(u) - \lambda &= l\\
                                   \lambda &\geq 0 \\
        \frac{1}{\epsilon}(u-g)+\lambda &\geq 0 \\
        \lambda\left(\frac{1}{\epsilon}(u-g)+\lambda\right) &= 0
           \label{strongform}
    \end{aligned}
    \quad \right\} \quad & \text{in $\Omega$,} \\[0.3cm]
        u=0\  \mbox{ and } \  
        \frac{\partial u}{\partial n} =0 \quad \mbox{ on } \partial \Omega.
 \end{align}
 
\begin{rem}
In case of a rigid obstacle, the first two equations in $(\ref{strongform})$  remain the same and the last two reduce to
\[
u-g\geq 0 \ \ {\rm in}\ \Omega, \qquad  \lambda(u-g) = 0  \ \ {\rm in}\ \Omega.
\]
\end{rem}

Defining the  bilinear and linear forms $\Bf:(V \times Q)\times(V \times Q)\rightarrow \mathbb{R}$ and \mbox{$\Lf:V \times Q\rightarrow \mathbb{R}$} through
 \begin{align}
    \Bf(w,\xi;v,\mu) &= a(w,v)-\langle v, \xi \rangle - \langle w, \mu \rangle - \epsilon \langle \xi, \mu \rangle, \\
    \Lf(v,\mu) &= (f,v) - \langle g, \mu \rangle,
\end{align}
the variational problem (\ref{weakform1})-(\ref{weakform2})  can be reformulated as 
\begin{prob}[Variational formulation]
Find $(u,\lambda) \in V \times \Lam$ such that
\begin{equation}
    \Bf(u,\lambda;v,\mu-\lambda) \leq \Lf(v,\mu-\lambda) \quad \forall (v,\mu) \in V \times \Lam.
    \label{bigvarform}
\end{equation}
\end{prob}

In the sequel, we will use the following norm in $V\times Q$
\begin{equation}
    \vertiii{(w,\xi)} = \big(\|w\|_2^2 + \|\xi\|_{-2}^2 + \epsilon \, \|\xi\|_0^2\big)^{1/2},
\end{equation}
with respect to which the bilinear form $\Bf$ is continuous.  Moreover, we write $a \gtrsim b$ (or $a \lesssim b$) when $a \geq C b$ (or $a \leq C b$) for some positive constant $C$ independent of the finite element mesh and of the parameter $\epsilon$.

\begin{thm}[Continuous stability]
    For every $(v,\mu) \in V \times Q$ there exists $w \in V$ such that
    \begin{equation}
        \Bf(v,\mu;w,-\mu) \gtrsim \vertiii{(v,\mu)}^2 \quad \text{and} \quad \|w\|_2 \lesssim \vertiii{(v,\mu)}.
        \label{contstab}
    \end{equation}
\end{thm}
\begin{proof}
	Define $p \in V$ through
	\begin{equation}
		a(p,q) = \langle q, \mu \rangle \quad \forall q \in V.
	\end{equation}
	From the continuity of the bilinear form $a$ it follows that
	\begin{equation}
		\frac{\langle q, \mu \rangle}{\|q\|_2} = \frac{a(p,q)}{\|q\|_2} \lesssim \|p\|_2 \quad \forall q \in V.
	\end{equation}
	Since $q$ is arbitrary, we have
	\begin{equation}
		\label{cstab:dir1}
		\|\mu\|_{-2} = \sup_{q \in V} \frac{\langle q, \mu \rangle}{\|q\|_2} \lesssim \|p\|_2.
	\end{equation}
	Moreover, the coercivity of the bilinear form $a$ gives
	\begin{equation}
		\label{cstab:dir2}
		\|p\|_2^2 \lesssim a(p,p) = \langle p, \mu \rangle \leq \|\mu\|_{-2} \|p\|_2 \quad \Rightarrow \quad \|p\|_2 \lesssim \|\mu\|_{-2}.
	\end{equation}
   Choosing $w = v-p$, noting that
   \begin{align*}
    \Bf(v,\mu;v-p,-\mu) & = a(v,v)-\langle v, \mu \rangle+\langle p, \mu \rangle+\epsilon \langle \mu, \mu \rangle \\ 
    & = \frac{1}{2} \big( a(v,v)+a(p,p)\big) + \frac{1}{2} a(v-p,v-p)+a(p,p) +\epsilon\langle\mu,\mu\rangle 
    \end{align*}
       and applying inequalities \eqref{cstab:dir1} and \eqref{cstab:dir2} proves the result.
\end{proof}

\section{The finite element method}

\label{sec:fem}

Let $\Ch$ be a conforming shape regular triangulation of $\Omega$ which we assume to be polygonal.
The finite element subspaces are
\begin{equation}
    V_h \subset V, \quad Q_h \subset Q.
\end{equation}
Moreover, we define
\begin{equation}
    \Lam_h = \{ \mu_h \in Q_h : \mu_h \geq 0 \text{~in $\Omega$}\} \subset \Lam.
\end{equation}
Let us introduce the stabilised bilinear and linear forms $\Bh$ and $\Lh$ by
\begin{align}
    \Bh(w,\xi;v,\mu) &= \Bf(w,\xi;v,\mu) - \alpha \sum_{K \in \Ch} h_K^4(\A(w)-\xi,\A(v)-\mu)_K, \\
    \Lh(v,\mu) &= \Lf(v,\mu) - \alpha \sum_{K \in \Ch} h_K^4(f,\A(v)-\mu)_K,
\end{align}
where $\alpha > 0$ is the stabilisation parameter.
\begin{prob}[The stabilised method]
    \label{prob:stab}
    Find $(u_h,\lambda_h) \in V_h \times \Lam_h$ such that
    \begin{equation}
        \Bh(u_h,\lambda_h;v_h,\mu_h-\lambda_h) \leq \Lh(v_h,\mu_h-\lambda_h) \quad \forall (v_h,\mu_h) \in V_h \times \Lam_h.
        \label{stabform}
    \end{equation}
\end{prob}
For the existence of a unique solution to Problem \ref{prob:stab}, see, {\sl e.g.}, \cite{BHR78}.

Let us define the  mesh-dependent norms
\begin{align}
    \|\xi_h\|_{-2,h}^2 &= \sum_{K \in \Ch} h_K^4 \|\xi_h\|_{0,K}^2, \\
    \vertiii{(w_h,\xi_h)}_h^2 &= \|w_h\|_2^2 + \|\xi_h\|_{-2}^2+ \|\xi_h\|_{-2,h}^2 + \epsilon \|\xi_h\|_0^2,
\end{align}
and recall the following  estimate.
\begin{lem}[Inverse inequality]
    There exists $C_I>0$ such that
    \begin{equation}
        C_I \|\A(w_h)\|_{-2,h}^2 \leq a(w_h,w_h) \quad \forall w_h \in V_h.
            \label{inverse}
    \end{equation}
\end{lem}
    The inverse estimate of the following lemma is valid in an arbitrary piecewise polynomial finite element space $Q_h$.
\begin{lem} It holds that
\begin{equation} 
\|\xi_h\|_{-2,h}  \lesssim \|\xi_h \|_{-2} \ \ \forall \xi_h\in Q_h\, .
\label{neginverse}
\end{equation}
\end{lem}
\begin{proof} Let $b_K\in P_6(K)$ be the sixth order bubble function
   \begin{equation}
        b_K = (\lambda_{1,K} \lambda_{2,K} \lambda_{3,K})^2 ,
    \label{bubble}
    \end{equation}
    where $\lambda_{j,K}, j\in\{1,2,3\}$, denote the barycentric coordinates for $K\in\Ch$, and
define the auxiliary space 
\[
W_h=\{ \, v_h\in H^2_0(\Omega)\, \vert \, v_h\vert _K=b_K \xi_h \vert_K,  \ \xi_h\in Q_h\, \}.
\]

Given  $\xi\in Q_h$, we now define $v_h\in W_h$ by
\[
v_h\vert _K =h_K^4 b_K \xi_h\vert _K, \quad K\in \Ch.
\]
From the norm equivalence and the inverse estimates, it follows that
\[
(v_h,\xi_h)\gtrsim \|\xi_h\|_{-2,h} ^2
\]
and
\[
\Vert v_h\Vert_2 \lesssim \vert v_h\vert_2\lesssim \|\xi_h\|_{-2,h} .
\]
Therefore
\[
\|\xi_h\|_{-2,h}  \lesssim \frac{(v_h,\xi_h)}{ \Vert v_h\Vert_2}
\]
and the assertion follows from the definition of the negative norm \eqref{negatnorm}. 
\end{proof}

For the proof of the following result, we refer to \cite{GSV16} (with minor modifications). 
\begin{lem}\label{pvtrick} There exist positive constants $C_1$ and $C_2$ such that 
 \begin{equation}\label{aux}
 \sup_{v_h\in V_h} \frac{\langle  v_h, \xi_h\rangle   }{\Vert   v_h\Vert_{2}} \ge C_1 \Vert \xi_h\Vert_{-2} -C_2 \Vert \xi_{h}\Vert_{-2,h} \quad \forall \xi_h \in Q_h.
 \end{equation}
 \end{lem}

\begin{thm}[Discrete stability]
    Suppose that $0 < \alpha < C_I$. It holds: for all $(v_h,\mu_h) \in V_h \times Q_h$ there exists $w_h \in V_h$ such that
    \begin{equation}
        \Bh(v_h,\mu_h;w_h,-\mu_h) \gtrsim \vertiii{(v_h,\mu_h)}_h^2 \quad \text{and} \quad \vertiii{(w_h,-\mu_h)}_h \lesssim \vertiii{(v_h,\mu_h)}_h.
        \label{discrstab}
    \end{equation}
\end{thm}
\begin{proof}
  In view of the inverse inequality (\ref{inverse}), it holds 
    \begin{equation}
        \begin{aligned}
            &\Bh(v_h,\mu_h;v_h,-\mu_h) \\
            &= a(v_h,v_h) + \epsilon \|\mu_h\|_0^2 - \alpha \|\A(v_h)\|_{-2,h}^2 + \alpha \|\mu_h\|_{-2,h}^2 \\
            &\geq \left(1-\alpha C_I^{-1}\right)a(v_h,v_h) + \min\{1,\alpha\}\left(\|\mu_h\|_{-2,h}^2+\epsilon \|\mu_h\|_0^2\right). 
        \end{aligned}
        \label{aux1}
    \end{equation}
   Let $q_h \in V_h$ be the function corresponding to the supremum in Lemma \ref{pvtrick},  scaled in such a way that $\|q_h\|_2=\|\mu_h\|_{-2}$. Then
    \begin{equation}
        \begin{aligned}
            &\Bh(v_h,\mu_h;-q_h,0) \\
            &= -a(v_h,q_h) + \langle q_h, \mu_h \rangle + \alpha \sum_{K \in \Ch} h_K^4 (\A(v_h)-\mu_h, \A(q_h) )_{0,K} \\
            &\geq -\| v_h\|_2 \|q_h\|_2 + C_1 \|\mu_h\|_{-2}\|q_h\|_{2}  - C_2 \|\mu_h\|_{-2,h} \|q_h\|_{2}\\ 
            & \qquad \qquad \qquad - \alpha \left(\|\A(v_h)\|_{-2,h} + \|\mu_h\|_{-2,h}\right) \|\A(q_h)\|_{-2,h}.
        \end{aligned}
    \end{equation}
 Using again the inverse inequality \eqref{inverse}, Young's inequality and the continuity of the bilinear form $a$, we conclude that
\begin{equation}
\Bh(v_h,\mu_h;-q_h,0) \geq C_3 \|\mu_h\|_{-2}^2 - C_4 \big(a(v_h,v_h)+\|\mu_h\|_{-2,h}^2\big) .
\label{aux2}
\end{equation}
Finally, taking $w_h=v_h-\delta q_h$ and using estimates \eqref{aux1} and \eqref{aux2}, together with the coercivity of $a$ and the assumption $0 < \alpha < C_I$, proves the stability bound after $\delta>0$ is chosen small enough. 

The estimate $\vertiii{(w_h,-\mu_h)} \lesssim \vertiii{(v_h,\mu_h)}$ is trivial and the same bound in the  discrete norm follows from the inverse estimate \eqref{neginverse}. 
\end{proof}

\begin{rem}\label{dsrem}
Note that the discrete stability bounds \eqref{discrstab} are also valid in the continuous norm $\vertiii{(\cdot,\cdot)}$.
\end{rem}

In the sequel, our functions may belong to the space $H^{-2}(\omega),~\omega\subset\Omega,$ equipped with the norm
\begin{equation}
    \|\mu\|_{-2,\omega} = \sup_{z \in H^2_0(\omega)} \frac{\langle z, \mu \rangle}{\|z\|_{2,\omega}}.
\end{equation}
This means that if $w \in H^2_0(\Omega)$ is such that $w|_\omega \in H^2_0(\omega)$ and $w = 0$ in $\Omega \setminus \omega$, we can write
\begin{equation}
    \label{eq:localdualnormineq}
    \langle w, \mu \rangle \leq \|\mu\|_{-2,\omega} \|w\|_{2,\omega}, \quad \forall \mu \in H^{-2}(\omega).
\end{equation}
Let $f_h \in V_h$ be the $L^2$ projection of $f$ and define the data oscillation as
\begin{align}
    \osc_K(f) &= h_K^2 \| f - f_h \|_{0,K}, \\
    \osc(f)^2 &= \sum_{K \in \Ch} \osc_K(f)^2.
\end{align}
Furthermore, we recall the following
integration by parts formula (cf. \cite{FS}), valid in any domain $R\subset \Omega$  
\begin{equation}\begin{aligned}
             \displaystyle    a_R(w,v) =  & \int_\sub \A(w)\,v  \,\mathrm{d}x    + 
           \int_{\partial \sub} \qn(w)v \,\mathrm{d}s \\
           &- \int_{\partial \sub} \, \Big(\mnn(w)\frac{\partial v}{\partial n}+    \mns(w)\frac{\partial v}{\partial s}\,\Big)\, \mathrm{d}s.
   \end{aligned}
   \label{intparts}
\end{equation}
  where we have used the shorthand notation
\[
   a_R(w,v) 
         = \int_R \Mom(w) : \Kappa(v) \,\mathrm{d}x, 
\]   
   and  defined the normal shear force and the normal and twisting moments  through
\[
\begin{aligned}
&\qn(w)=\Shear(w)\cdot \boldsymbol{n},\ \ \ \mnn(w) =  \boldsymbol{n} \cdot  \Mom(w) \boldsymbol{n}, \\  &\mns(w) =\msn(w)=
 \boldsymbol{s} \cdot  \Mom(w) \boldsymbol{n} ,
 \end{aligned}
\]
with $\boldsymbol{n}$ and $\boldsymbol{s}$ denoting the normal and tangential  directions at $\partial \sub$.
Integrating by parts on a smooth subset $S \subset R$ we get
\begin{equation}
    \label{subsetintparts}
    \int_{S} Q_n(w) v\,\mathrm{d}s -\int_{S} M_{ns}(w) \frac{\partial v}{\partial s} \, \mathrm{d}s = \int_S V_n(w)\,\mathrm{d}s - \Big|_a^b M_{ns}(w)v,
\end{equation}
where $a$ and $b$ are the endpoints of $S$ and
the quantity
\begin{equation}\label{ksf}
\effs(w)= \qn(w) +\frac{ \partial \mns(w)}{\partial s} 
\end{equation}
is called the  {\em Kirchhoff shear force} (cf.~\cite{FS}).
Denote by $\omega_E=K_1\cup K_2$ the pair of triangles sharing an edge $E$ and  define  jumps in the normal moment and the shear force over $E$ through
\[\begin{aligned}
  \llbracket  \mnn(v)  \rrbracket \vert_E &=   \mnn(v)  - M_{n'n'}(v)  \\
  \llbracket  \Vn(v)  \rrbracket \vert_E &=   \Vn(v)   +V_{n'}(v).
  \end{aligned}  \]
where $\boldsymbol{n}$ and $\boldsymbol{n}'$ stand for the outward normals to $K_1$ and $K_2$, respectively.

We will need the following lemma in proving the a priori and a posteriori estimates. We will sketch its proof and refer to \cite{GSV18} for more details.

\begin{lem}
    For all $v_h \in V_h$ and $\mu_h \in Q_h$ it holds that
    \begin{align}
        h_K^2 \| \A(v_h) - \mu_h - f\|_{0,K} &\lesssim \|u-v_h\|_{2,K} + \|\lambda -\mu_h\|_{-2,K} + \osc_K(f), \label{eq:lowboundlocalint} \\
        h_E^{1/2} \| \llbracket M_{nn}(v_h) \rrbracket \|_{0,E} &\lesssim \|u-v_h\|_{2,\omega_E} + \|\lambda - \mu_h\|_{-2,\omega_E} + \sum_{K \subset \omega_E} \osc_K(f),\label{eq:lowboundlocaledge1}\\
        h_E^{3/2} \| \llbracket V_{n}(v_h) \rrbracket \|_{0,E} &\lesssim \|u-v_h\|_{2,\omega_E} + \|\lambda - \mu_h\|_{-2,\omega_E} + \sum_{K \subset \omega_E} \osc_K(f).\label{eq:lowboundlocaledge2}
    \end{align}
\end{lem}
\begin{proof}
 Recall from \eqref{bubble} the sixth order bubble $b_K \in P_6(K)$ 
    and let 
    \[
        z_K = b_K h_K^4(\A(v_h)-\mu_h-f_h),
    \]
    for every $(v_h,\mu_h) \in V_h \times \varLambda_h$.  Testing with
    $z_K$ in the continuous variational problem (\ref{weakform1}) gives the identity
    \[
        a_K(u,z_K) - \langle z_K, \lambda \rangle = (f,z_K)_K.
    \]
  
    We have
    \begin{equation}
        \label{eq:lemp1}
        \begin{aligned}
            &h_K^4 \|\A(v_h)-\mu_h-f_h\|_{0,K}^2 \\
            &\lesssim h_K^4 \| \sqrt{b_K}(\A(v_h)-\mu_h-f_h) \|_{0,K}^2 \\
            &= (\A(v_h)-\mu_h-f_h,z_K)_K \\
            &= (\A(v_h)-\mu_h,z_K)_K - (f,z_K)_K + (f-f_h,z_K)_K \\ 
            &= a_K(v_h-u,z_K)+\langle z_K, \lambda-\mu_h \rangle + (f-f_h,z_K)_K. 
        \end{aligned}
    \end{equation}
    The bound \eqref{eq:lowboundlocalint} follows from the continuity of
    $a$, Cauchy--Schwarz and inverse inequalities and from inequality \eqref{eq:localdualnormineq}.

    Following \cite{GudiPorwal2016}, see also \cite{GSV18}, we let $\omega_E = K_1 \cup K_2$ and define an auxiliary function $w = p_1 p_2 p_3$ in such  a way that
    \begin{itemize}
        \item $p_1$ is the extension of $\llbracket M_{nn}(v_h) \rrbracket$ to $\omega_E$ such that $\tfrac{\partial p_1}{\partial n_E} = 0$;
        \item $p_2$ is the eight order bubble that, together its first order derivatives, vanishes at $\partial \omega_E$ and equals to one at the midpoint of $E$;
        \item $p_3$ is the linear polynomial that is zero on $E$ and satisfies $\tfrac{\partial p_3}{\partial n_E} = 1$.
    \end{itemize}
    Outside of $\omega_E$,  $w$ is extended by zero, see  \cite{GSV18} for more
    details. From
    the construction of $w$ and  formula (\ref{intparts}), it follows that
    \begin{align*}
        \|\llbracket M_{nn}(v_h) \rrbracket\|_{0,E}^2 &\lesssim (M_{nn}(v_h), \tfrac{\partial w}{\partial n_E} )_E \\
                                                      &= - a_{\omega_E}(v_h, w) + (\A(v_h), w)_{\omega_E} \\
                                                      &= a_{\omega_E}(u - v_h, w) + (\A(v_h) - \mu_h - f, w)_{\omega_E} + \langle w, \mu_h - \lambda \rangle.
    \end{align*}
   Bound
    \eqref{eq:lowboundlocaledge1}  can now be established  using the continuity of the bilinear form $a$,
    the Cauchy--Schwarz and inverse inequalities, a scaling argument and inequalities
    \eqref{eq:localdualnormineq} and \eqref{eq:lowboundlocalint}, see  \cite{GSV16} and \cite{GSV18} for similar considerations.

    The proof of \eqref{eq:lowboundlocaledge2} is similar except for the construction of the auxiliary
    function. We choose $w^\prime$ as a function
    defined on  a subset   of $\omega_E$,  consisting of two smaller triangles $K_1^\prime$ and $K_2^\prime$,   symmetric with respect to the edge $E$, and write
    $\omega_E^\prime = K_1^\prime \cup K_2^\prime$, ~$K_j^\prime \subset K_j$,~$j\in\{1,2\}$.
Then we define $w^\prime = p_1^\prime p_2^\prime$ where
    \begin{itemize}
        \item $p_1^\prime$ is an extension of $\llbracket V_n(v_h) \rrbracket$ to $\omega_E^\prime$
            such  that $\tfrac{\partial p^\prime_1}{\partial n_E} = 0$;
        \item $p_2^\prime$ is the eight order bubble that, together its first order derivatives,
            vanishes on $\partial \omega_E^\prime$ and equals to one at the midpoint of $E$.
    \end{itemize}
    Note that now due to symmetry $\tfrac{\partial w^\prime}{\partial n_E} \big|_E = 0$.
    Now, recalling identities \eqref{intparts} and \eqref{subsetintparts}, and integrating in parts in the last term on its right-hand side (cf. \cite{GSV18}), we obtain
    \begin{align*}
        \|\llbracket V_{n}(v_h) \rrbracket\|_{0,E}^2 &\lesssim (V_{n}(v_h), w^\prime )_E \\
                                                      &= - a_{\omega_E^\prime}(v_h, w^\prime) + (\A(v_h), w^\prime)_{\omega_E^\prime} \\
                                                      &= a_{\omega_E^\prime}(u - v_h, w^\prime) + (\A(v_h) - \mu_h - f, w^\prime)_{\omega_E^\prime} + \langle w^\prime, \mu_h - \lambda \rangle,
    \end{align*}
   from which estimate  \eqref{eq:lowboundlocaledge2} can be concluded as the final step for bound \eqref{eq:lowboundlocaledge1}.
\end{proof}

\begin{thm}[A priori estimate]
It holds that
    \begin{equation}
        \begin{aligned}
            &\vertiii{(u-u_h,\lambda-\lambda_h)} \\
            & \qquad \lesssim \inf_{\substack{v_h \in V_h, \\ \mu_h \in \Lambda_h}} \left( \vertiii{(u-v_h,\lambda-\mu_h)} + \sqrt{\langle u-g+\epsilon \lambda, \mu_h \rangle} \right) + \osc(f).
    \end{aligned}
    \label{aprioriest}
    \end{equation}
\end{thm}
\begin{proof}
    Let $(v_h,\mu_h) \in V_h \times Q_h$ be arbitrary and assume that $w_h \in V_h$ is the function corresponding to $(u_h-v_h,\lambda_h-\mu_h)$ in the discrete stability estimate (\ref{discrstab}) expressed in the continuous norm $\vertiii{(\cdot,\cdot)}$, see Remark \ref{dsrem}.  
The problem statement then implies that
    \begin{align*}
        \vertiii{(u_h-v_h,\lambda_h-\mu_h)}^2 &\lesssim \Bh(u_h-v_h,\lambda_h-\mu_h;w_h,\mu_h-\lambda_h) \\
                                              &\lesssim \Lh(w_h,\mu_h-\lambda_h) - \Bf(v_h,\mu_h,w_h,\mu_h-\lambda_h) \\
                                              &\phantom{=}+ \alpha \sum_{K \in \Ch} h_K^4(\A(v_h)-\mu_h,\A(w_h)+\lambda_h-\mu_h)_K \\
                                              &= \Bf(u-v_h,\lambda-\mu_h;w_h,\mu_h-\lambda_h) + \langle u-g+\epsilon \lambda, \mu_h-\lambda_h \rangle \\
                                              &\phantom{=} + \alpha \sum_{K \in \Ch} h_K^4(\A(v_h)-\mu_h-f,\A(w_h)+\lambda_h-\mu_h)_K.
    \end{align*}
    Let us  bound separately each term on the right hand side. 
    The continuity of the bilinear form $\Bf$ and the second estimate in (\ref{discrstab})  yield for the first term
    \begin{align*}
        \Bf(u-v_h,\lambda-\mu_h;w_h,\mu_h-\lambda_h) &\lesssim \vertiii{(u-v_h,\lambda-\mu_h)} \vertiii{(w_h,\mu_h-\lambda_h)} \\
                                                     &\lesssim \vertiii{(u-v_h,\lambda-\mu_h)} \vertiii{(u_h-v_h,\lambda_h-\mu_h)}.
    \end{align*}
    For the second term we obtain
    \[
        \langle u - g + \epsilon \lambda, \mu_h-\lambda_h \rangle \leq \langle u -g + \epsilon \lambda, \mu_h-\lambda \rangle = \langle u-g+\epsilon\lambda,\mu_h \rangle.
    \]
    The third term is bounded as follows
    \begin{align*}
        &\sum_{K \in \Ch} h_K^4(\A(v_h)-\mu_h-f,\A(w_h)+\lambda_h-\mu_h)_K \\
        &\leq \left(\sum_{K \in \Ch} h_K^4 \|\A(v_h)-\mu_h-f\|_{0,K}^2\right)^{1/2} \left(\sum_{K\in\Ch} h_K^4 \|\A(w_h)\|_{0,K}^2\right)^{1/2}\\
        &\phantom{=}+\left(\sum_{K \in \Ch} h_K^4 \|\A(v_h)-\mu_h-f\|_{0,K}^2\right)^{1/2} \left(\sum_{K \in \Ch} h_K^4\|\lambda_h-\mu_h\|_{0,K}^2\right)^{1/2} \\
        &\lesssim \left( \|u-v_h\|_2 + \|\lambda - \mu_h\|_{-2} + \osc(f) \right)\left( \sqrt{a(w_h,w_h)} + \|\lambda_h-\mu_h\|_{-2,h}\right) \\
        &\lesssim \left( \vertiii{(u-v_h,\lambda-\mu_h)} + \osc(f)\right) \vertiii{(u_h-v_h,\lambda_h-\mu_h)},
    \end{align*}
where we have used  \eqref{eq:lowboundlocalint}, the second estimate in (\ref{discrstab}) and  the inverse inequalities (\ref{inverse}) and  \eqref{neginverse}.
\end{proof}


To derive a posteriori error bounds, we define the local residual estimators
\begin{align}
    \eta_K^2 &= h_K^4 \| \A(u_h) - \lambda_h -f \|_{0,K}^2, \\
    \eta_{E}^2 &= h_E^3 \| \llbracket V_n(u_h) \rrbracket \|_{0,E}^2 + h_E \| \llbracket M_{nn}(u_h) \rrbracket \|_{0,E}^2,
\end{align}
and the corresponding global residual estimator 
\begin{equation}
    \eta^2 = \sum_{K \in \Ch} \eta_K^2 + \sum_{E \in \Ehint} \eta_E^2, 
\end{equation}
where $\Ehint$ denotes the set of interior edges in the mesh.
An additional global estimator $S$, due to the unknown location of the contact boundary, is defined through
\begin{equation}
    S^2 = 
   ((u_h-g+\epsilon\lambda_h)_+, \lambda_h) +  \sum_{K\in \mathcal{C}_{h}} \frac{1}{\epsilon+h_K^4} \| (g-u_h-\epsilon\lambda_h)_+\|_{0,K}^2
\end{equation}
 where $u_+=\max\{u,0\}$ denotes the positive part of $u$. 

The lower bound is a simple consequence of the global versions of estimates \eqref{eq:lowboundlocalint}--\eqref{eq:lowboundlocaledge2}. We refer to \cite{GSV16} for a similar consideration with more details.

\begin{thm}[A posteriori estimate -- efficiency]
    It holds that
    \begin{equation}
     \eta  \lesssim    \vertiii{(u-u_h,\lambda-\lambda_h)}.
    \end{equation}
\end{thm}

The upper bound cannot be established as elegantly as for the second-order  (membrane) obstacle problem,  cf. \cite{GSV16}, since the positive part function is not in $H^2(\Omega)$. We will use the following  assumption, justified by the a priori estimate \eqref{aprioriest} for regular enough solution, see, {\sl e.g.},  \cite{CFHPR18}.  
\begin{ass}[Saturation assumption] There exists $\beta <1$ such that
\[
  \vertiii{(u-u_{h/2},\lambda-\lambda_{h/2})}_{h/2}  \leq \beta   \vertiii{(u-u_{h},\lambda-\lambda_{h})}_{h} ,
  \]
    where $(u_{h/2},\lambda_{h/2})\in V_{h/2}\times Q_{h/2}$ is the solution in the mesh $\mathcal{C}_{h/2}$ obtained by splitting the elements of the mesh  $\mathcal{C}_{h}$.
\label{satass}
\end{ass}

\begin{thm}[A posteriori estimate -- reliability]
    It holds that
    \begin{equation}
        \vertiii{(u-u_h,\lambda-\lambda_h)} \lesssim \eta+S.
    \end{equation}
\end{thm}

\begin{proof}
 Let $w\in V_{h/2}$ be the function corresponding to $(u_{h/2}-u_h,\lambda_{h/2}-\lambda_h)\in V_{h/2}\times Q_{h/2}$ in the discrete stability estimate \eqref{discrstab} for which is holds, in particular, that
 \begin{equation}
 \|w\|_2 \lesssim \vertiii{(u_{h/2}-u_h,\lambda_{h/2}-\lambda_h)}_{h/2} .
 \label{waux}
 \end{equation}
  Let, moreover, $\widetilde{w}\in V_h$ denote the Hermite type interpolant of $w\in V_{h/2}$. By scaling, one readily shows that
\begin{equation}\begin{array}{l}
 \displaystyle   \sum_{K \in \Ch} h_K^{-4} \|w-\widetilde{w}\|_{0,K}^2 +\sum_{E \in \Ehint} h_E^{-1}\|\nabla(w-\widetilde{w})\|_{0,E}^2\\
   \displaystyle   \quad+\sum_{E \in \Ehint}  h_E^{-3}\|w-\widetilde{w}\|_{0,E}^2 \lesssim \|w\|_2^2 \qquad \text{and} \qquad \|\widetilde{w}\|_2 \lesssim  \|w\|_2.
  \end{array}\label{hermite}
  \end{equation}

    The discrete problem statement implies that
    \begin{equation}
        0 \leq - \Bh(u_h,\lambda_h;-\widetilde{w},0) + \Lh(-\widetilde{w},0).
        \label{discraux}
    \end{equation}
   From \eqref{discrstab}, \eqref{stabform} and \eqref{discraux} it then follows that
 \begin{align*}
\vertiii{(u_{h/2}-u_h,\lambda_{h/2}-\lambda_h)}_{h/2}^2 & \lesssim 
\mathcal{B}_{h/2}(u_{h/2}-u_h,\lambda_{h/2}-\lambda_h;w,\lambda_h-\lambda_{h/2}) \\ &  \leq 
\mathcal{L}_{h/2}(w,\lambda_h-\lambda_{h/2}) -\mathcal{B}_{h/2}(u_h,\lambda_h;w,\lambda_h-\lambda_{h/2})\\ \qquad &\quad  - \mathcal{B}_h(u_h,\lambda_h;-\widetilde{w},0)+\mathcal{L}_h(-\widetilde{w},0) \\ &= (f,w-\widetilde{w}) 
 -a(u_h,w-\widetilde{w})+\langle w-\widetilde{w},\lambda_h\rangle \\
 & \quad + \langle u_h+\epsilon\lambda_h-g,\lambda_h-\lambda_{h/2}\rangle\\ 
& \quad +\alpha \sum_{K^\prime\in \mathcal{C}_{h/2}} h_{K^\prime}^4 \left( \mathcal{A}(u_h)-\lambda_h-f,\mathcal{A}(w)\right)_{K^\prime}
\\ &\quad - \alpha \sum_{K^\prime\in \mathcal{C}_{h/2}} h_{K^\prime}^4 \left( \mathcal{A}(u_h)-\lambda_h-f,\lambda_h-\lambda_{h/2}\right)_{K^\prime}
\\ &\quad  - \alpha
 \sum_{K\in \mathcal{C}_{h}} h_K^4 \left( \mathcal{A}(u_h)-\lambda_h-f,\mathcal{A}(\widetilde{w})\right)_K
\end{align*}

Using formula \eqref{intparts} to integrate by parts in $a(u_h,w-\widetilde{w})$, we obtain
    \begin{align*}
        & (f,w-\widetilde{w}) - a(u_h,w-\widetilde{w}) + \langle  w-\widetilde{w},\lambda_h \rangle \\
                 &= \sum_{K \in \Ch} \left(-\A(u_h)+\lambda_h+f,w-\widetilde{w}\right)_K \\
                   &\phantom{=} + \sum_{E \in \Ehint} \Big\{ \left(\llbracket \Mom(u_h) \boldsymbol{n} \rrbracket,\nabla(w-\widetilde{w})\right)_E 
      - (\llbracket \Shear(u_h)\cdot \boldsymbol{n} \rrbracket,w-\widetilde{w})_E \Big\} \\
        &= \sum_{K \in \Ch} \left(-\A(u_h)+\lambda_h+f,w-\widetilde{w}\right)_K \\        &\phantom{=} + \sum_{E \in \Ehint} \Big\{ \left(\llbracket M_{nn}(u_h) \rrbracket,\nabla(w-\widetilde{w}) \cdot \boldsymbol{n}\right)_E  - (\llbracket V_n(u_h) \rrbracket,w-\widetilde{w})_E \Big\} .
    \end{align*}
    These terms are easily bounded using the Cauchy--Schwarz inequality and the interpolation estimates \eqref{hermite}. 
    
    On the other hand,  dividing $u_h+\epsilon\lambda_h-g$ into its positive and negative part, we obtain the estimate
    \begin{align*}
 &\langle u_h+\epsilon\lambda_h-g, \lambda_h-\lambda_{h/2}\rangle \\
 & \quad  \leq \big((u_h+\epsilon\lambda_h-g)_+, \lambda_h\big) +
 \big( (u_h+\epsilon\lambda_h-g)_-, \lambda_h-\lambda_{h/2}\big) \\
 & \quad \leq  \big( (u_h+\epsilon\lambda_h-g)_+, \lambda_h\big) \\ 
 & \qquad + \left( \sum_{K\in \mathcal{C}_{h}} \frac{1}{\epsilon+h_K^4} \| (u_h+\epsilon\lambda_h-g)_-\|_{0,K}^2\right)^{1/2} \, 
\vertiii{(u_{h/2}-u_h,\lambda_{h/2}-\lambda_h)}_{h/2}
  \end{align*}

For the stabilising terms, we obtain the bounds
    \begin{align*}
        &\sum_{K \in \Ch} h_K^4(-\A(u_h)+\lambda_h+f,\A(\widetilde{w}))_K \\
        &\leq \sum_{K \in \Ch} h_K^4 \|\A(u_h)-\lambda_h-f\|_{0,K} \|\A(\widetilde{w})\|_{0,K} \\
                                                                                &\leq \left(\sum_{K \in \Ch} h_K^4 \|\A(u_h)-\lambda_h-f\|_{0,K}^2\right)^{1/2} 
                                                                                \left(\sum_{K \in \Ch} h_K^4 \|\A(\widetilde{w})\|_{0,K}\right)^{1/2} \\
                                                                                &\lesssim \left(\sum_{K \in \Ch} h_K^4 \|\A(u_h)-\lambda_h-f\|_{0,K}^2\right)^{1/2} \|w\|_2\\
                                                                                    &\lesssim \left(\sum_{K \in \Ch} h_K^4 \|\A(u_h)-\lambda_h-f\|_{0,K}^2\right)^{1/2} \vertiii{(u_{h/2}-u_h,\lambda_{h/2}-\lambda_h)}_{h/2},
       \end{align*}      
             \begin{align*}                                                                   
      &\sum_{K^\prime\in \mathcal{C}_{h/2}} h_{K^\prime}^4 \left( \mathcal{A}(u_h)-\lambda_h-f,\A(\widetilde{w})-(\lambda_h-\lambda_{h/2})\right)_{K^\prime}\\
        &\lesssim \left(\sum_{K \in \Ch} h_K^4 \|\A(u_h)-\lambda_h-f\|_{0,K}^2\right)^{1/2} \vertiii{(u_{h/2}-u_h,\lambda_{h/2}-\lambda_h)}_{h/2},
    \end{align*}
where we have used the inverse inequality \eqref{inverse} and the interpolation estimates \eqref{hermite}.   

 The assertion follows after completing the square, using again the estimates \eqref{hermite} and   \eqref{waux} and observing that
 \[
  \vertiii{(u-u_h,\lambda-\lambda_h)}  \leq   \vertiii{(u-u_h,\lambda-\lambda_h)}_h \leq \frac{1}{1-\beta} \vertiii{(u_{h/2}-u_h,\lambda_{h/2}-\lambda_h)}_{h/2} .
  \]
\end{proof}

\section{A practical solution algorithm}

\label{sec:nitsche}

The approximation properties of the primal variable and
the Lagrange multiplier are balanced when the polynomial order of the latter 
is four degrees smaller than that of the
displacement variable, for example, when the Argyris element is
 coupled with a piecewise linear and discontinuous approximation of the Lagrange multiplier. It is,
however, unnecessary to actually solve for the Lagrange multiplier since it
can be eliminated from the stabilised formulation altogether. This approach is
analogous to the derivation of  Nitsche's method for Dirichlet boundary
conditions (cf.~\cite{Stenberg75}) and hence we refer to the proposed
method as \emph{Nitsche's method for the Kirchhoff plate obstacle problem}.

Nitsche's method can be derived in two steps.  First, testing with $(0,-\mu_h)$ in the  stabilised formulation \eqref{stabform},
leads to the following elementwise expression for the Lagrange multiplier
\begin{equation*}
    \lambda_h|_K = \frac{1}{\epsilon+ \alpha h_K^4} \left(\pi_h g|_K - \pi_h u_h|_K + \alpha h_K^4(\pi_h \A(u_h) - \pi_h f)|_K \right)_+, \quad \forall K \in \Ch,
\end{equation*}
where $\pi_h : L^2(\Omega) \rightarrow Q_h$ is the $L^2$ projection.
Let the function $\mathcal{H} \in L^2(\Omega)$ be such that $\mathcal{H}|_K = h_K$, $\forall K \in \Ch$.
Then, testing with $(v_h,0)$, substituting  the formula for $\lambda_h$ in the resulting  expression and
choosing $Q_h = L^2(\Omega)$,  gives the following nonlinear variational problem:

\begin{prob}[Nitsche's method for Problem 1]
    \label{prob:nitsche}
    Find $u_h \in V_h$ such that
    \begin{equation}
        a_h(u_h, v_h; u_h) = l_h(v_h; u_h), \quad \forall v_h \in V_h,
        \label{nitsche-eqs}
    \end{equation}
    where
    \begin{align*}
        a_h(u_h,v_h; w_h)&=a(u_h,v_h) + \left(\tfrac{1}{\epsilon+ \alpha \mathcal{H}^4} u_h, v_h \right)_{\Omega_C(w_h)} - \left(\tfrac{\alpha \mathcal{H}^4}{\epsilon+ \alpha \mathcal{H}^4} \A(u_h), v_h\right)_{\Omega_C(w_h)} \\
                         &\quad- \left(\tfrac{\alpha \mathcal{H}^4}{\epsilon+ \alpha \mathcal{H}^4} u_h, \A(v_h)\right)_{\Omega_C(w_h)} - \left(\tfrac{\epsilon\alpha \mathcal{H}^4}{\epsilon+ \alpha \mathcal{H}^4} \A(u_h), \A(v_h) \right)_{\Omega_C(w_h)} \\
                         &\quad- \left(\alpha \mathcal{H}^4 \A(u_h), \A(v_h)\right)_{\Omega \setminus \Omega_C(w_h)}, \\
        l_h(v_h; w_h)&= (f,v_h) + \left( \tfrac{1}{\epsilon+ \alpha \mathcal{H}^4} g, v_h\right)_{\Omega_C(w_h)} - \left(\tfrac{\alpha \mathcal{H}^4}{\epsilon+ \alpha \mathcal{H}^4} g, \A(v_h)\right)_{\Omega_C(w_h)} \\
                     &\quad - \left( \tfrac{\alpha \mathcal{H}^4}{\epsilon+ \alpha \mathcal{H}^4} f, v\right)_{\Omega_C(w_h)} - \left( \tfrac{\epsilon\alpha \mathcal{H}^4}{\epsilon+ \alpha \mathcal{H}^4} f, \A(v_h) \right)_{\Omega_C(w_h)} \\
                     &\quad - (\alpha \mathcal{H}^4 f, \A(v_h))_{\Omega \setminus \Omega_C(w_h)}.
    \end{align*}
  The contact set $   \Omega_C(w_h) $ above is defined as
    \begin{equation*}
        \Omega_C(w_h) = \{ (x,y) \in \Omega : F(w_h) > 0 \},
    \end{equation*}
    with $F(w_h) $ denoting the reaction force  given by
    \begin{equation*}
        F(w_h) = \frac{1}{\epsilon+ \alpha \mathcal{H}^4} \left(g - w_h + \alpha \mathcal{H}^4(\A(w_h) - f) \right)_+.
    \end{equation*}
\end{prob}

The practical solution algorithm for Problem~\ref{prob:nitsche} is an iterative
process where at each step the contact set $\Omega_C$ is approximated using the
displacement field from the previous iteration so that system \eqref{nitsche-eqs} becomes linear.
The  process is terminated as soon as  the norm of the displacement field
is below a predetermined tolerance $TOL>0$. The stopping criterion is formulated
with respect to the strain energy norm
\begin{equation}
    \|w\|_E = \sqrt{a(w,w)}.
\end{equation}

\begin{algorithm}
    \caption{Nitsche's method, with contact iterations}
    \label{alg:nitsche}
    \begin{algorithmic}[1]
        \STATE $k \leftarrow 0$
        \WHILE{$k<1$ or $\|u_h^{k} - u_h^{k-1}\|_E \leq TOL$}
        \STATE Find $u_h^{k+1} \in V_h$ s.t.~$a_h(u_h^{k+1},v_h;u_h^k) = l_h(v_h;u_h^k)$, $\forall v_h \in V_h$.
        \STATE $k \leftarrow k + 1$
        \ENDWHILE
        \RETURN $u_h^{k}$
    \end{algorithmic}
\end{algorithm}

For a discussion regarding the  convergence of  iterations in
Algorithm~\ref{alg:nitsche}, we refer to \cite{GSV16} where we compare this approach to the semismooth Newton
method for solving a stabilised second-order obstacle problem. We point out that the semismooth Newton
method (see e.g.~\cite{Wohlmuth2011}) corresponds to an algorithm, similar to Algorithm~\ref{alg:nitsche},
where the contact area follows element boundaries. Hence, we
expect Algorithm~\ref{alg:nitsche} to behave numerically  as a
semismooth \mbox{Newton-type} strategy applied to variational inequalities.

For an adaptive refinement, we use the maximum strategy with the
parameter $\theta \in (0,1)$ for marking and the red-green-blue refinement, see
e.g.~\cite{VerfurthBook,bartels2015numerical}. The error estimator is defined as
\begin{equation}
    \mathcal{E}_K^2 = \eta_K^2 + \frac12 \sum_{E \subset K} \eta_E^2 + ((u_h - g + \epsilon \lambda_h)_+, \lambda_h)_K + S_{K,\epsilon }^2,
\end{equation}
where
\begin{equation}
    S_{K,\epsilon } =   \frac{1}{\sqrt{\epsilon+h_K^4}} \| (g-u_h-\epsilon\lambda_h)_+\|_{0,K}
\end{equation}
Given the displacement field $u_h$, the reaction force $\lambda_h = F(u_h)$
is computed
as indicated in Problem~\ref{prob:nitsche}.
We start with an initial mesh $\Ch^0$ and terminate
the computation after a predetermined number of adaptive refinement steps $M$.
The resulting procedure is summarised in the listing Algorithm~\ref{alg:adaptive}.

\begin{algorithm}[h]
    \caption{The adaptive Nitsche's method}
    \label{alg:adaptive}
    \begin{algorithmic}[1]
        \STATE $j \leftarrow 0$
        \WHILE{$j < M$}
        \STATE Solve $u_h^{j+1}$ using Algorithm~\ref{alg:nitsche} and the mesh $\Ch^j$.
        \STATE Evaluate the error estimator $\mathcal{E}_K$ for every $K \in \Ch^j$.
        \STATE Using the red-green-blue refinement strategy~\cite{VerfurthBook,bartels2015numerical}, construct $\Ch^{j+1}$ by refining the elements $K$
        that satisfy the inequality
        \begin{equation*}
            \mathcal{E}_K > \theta \max_{K^\prime \in \Ch^j} \mathcal{E}_{K^\prime}.
        \end{equation*}
        \STATE $j \leftarrow j + 1$
        \ENDWHILE
    \end{algorithmic}
\end{algorithm}

\section{Numerical results}

\label{sec:numerical}

We illustrate the performance of the proposed algorithms by solving two example
problems and comparing the uniform and adaptive meshing. The adaptive method is expected to recover the optimal rate of convergence
with respect to the number of degrees of freedom $N$, i.e.
\[
   \vertiii{(u-u_h,\lambda-\lambda_h)} ~~\propto~~ N^{-\tfrac{k-1}{2}},
\]
where $k$ is the polynomial order of the finite element basis.
As a measure of error we use the global estimator $\eta + S$.
We expect that, asymptotically, it holds
\[
    \vertiii{(u-u_h,\lambda-\lambda_h)} ~~\propto~~ \eta + S.
\]

Let $\Omega = [0,1]^2$ and let $\Ch$ be a triangulation of $\Omega$.  The finite
element space for the displacement field  consists of a set of piecewise  polynomials of order five, i.e.
\[
    V_h = \{ w \in H^2_0(\Omega) : w|_{K} \in P_5(K)~\forall K \in \Ch \}.
\]
The global $C^1$-continuity 
is conceived  by implementing the Argyris basis functions, c.f.~\cite{CiarletBook}.
In both examples, the loading function and  the material parameters are chosen as
$\mbox{$f=-10$}$, $d=1$, $E=1$ and $\nu=0$.  For the parameters $\alpha$,
$TOL$ and $\theta$, we use the values  $\alpha = 10^{-5}$, $TOL = 10^{-10}$ and
$\theta=0.5$. In each case, we start with the  mesh shown in the upper left
panel of Figure~\ref{fig:meshrigid} and apply either a uniform refinement (each
triangle is split into four subtriangles) or Algorithm~\ref{alg:adaptive} with
$M=5$.

The first example is that of a rigid obstacle, $\epsilon = 0$, with
its shape defined by
\begin{equation}
    g(x,y) = -100\left((x-0.5)^2 + (y-0.5)^2\right).
\end{equation}
This obstacle is smooth and hence it belongs to $H^2(\Omega)$ as required by the continuous
formulation. Nevertheless, its shape is sharp due to the moderately large negative
coefficient. Qualitatively, the plate behaves subject to this type of obstacle as it
would under a point load and we expect the error estimator to be large
near the midpoint $(0.5;0.5)$. 

The resulting sequence of adaptive meshes is depicted in Figure~\ref{fig:meshrigid}
and the respective global errors can be found in Figure~\ref{fig:resultsrigid}. The discrete
solution and the Lagrange multiplier, after three adaptive refinements, are shown  in Figures~\ref{fig:discrete}(A) and \ref{fig:lagmult}(A). The discrete functions are visualised in a refined mesh as they may have
high-order and non-smooth variations inside the elements.  The results of
Figure~\ref{fig:resultsrigid} clearly indicate that the adaptive
method gains the optimal rate of convergence $\mathcal{O}(N^{-2})$ whereas the uniform
refinement is observed to be $\mathcal{O}(N^{-1/2})$. Note that if the numerical contact region was larger, for example using a less sharp obstacle, the convergence rate would become limited by the regularity of the solution,  i.e. with uniform refinement  eventually  by $\mathcal{O}(N^{-3/4})$. 

In the second example, we consider an elastic obstacle ($\epsilon > 0$) defined by the function
\begin{equation}
    g(x,y) = \begin{cases}
        0, & \text{if $(x,y) \in [0.3;0.7]^2$,} \\
        -1, & \text{otherwise.}
    \end{cases}
\end{equation}
Note that $g \in L^2(\Omega)$ but  \mbox{$g \not\in H^1(\Omega)$}.
Computing the cases $\epsilon = 10^{-j}$, $j \in \{3,4,5,6\}$, we observe that
the behaviour of the reaction force varies quite much from case to case as
revealed by the discrete Lagrange multipliers depicted in
Figure~\ref{fig:lagmult}
and by the discrete contact sets  shown in Figure~\ref{fig:contactset}.
In particular, the contact sets corresponding to less rigid obstacles remain simply
connected which is not the case for the stiffer obstacles.

The resulting error graphs for the adaptive and uniform
refinements can be found in Figure~\ref{fig:elastic}.
The sequences of adaptive meshes can be found in Figures~\ref{fig:meshe6}--\ref{fig:meshe3}.
We observe that for uniform refinements the slope of the error graph is
getting worse when the obstacle is stiffened and that the adaptive meshing
strategy successfully recovers the optimal rate of convergence $\mathcal{O}(N^{-2})$, independently of the
value of $\epsilon $.

\pgfplotstableread{
    ndofs estimator
    206 13.8202069961
    694 6.88266076402
    2534 3.41938921704
}\uniformsquare

\pgfplotstableread{
    ndofs estimator
    206 13.7709177269
    422 6.85949029861
    638 3.41300123948
    854 1.70023134328
    1070 0.852819557242
}\adaptivesquare

\begin{figure}
    \includegraphics[width=0.45\textwidth]{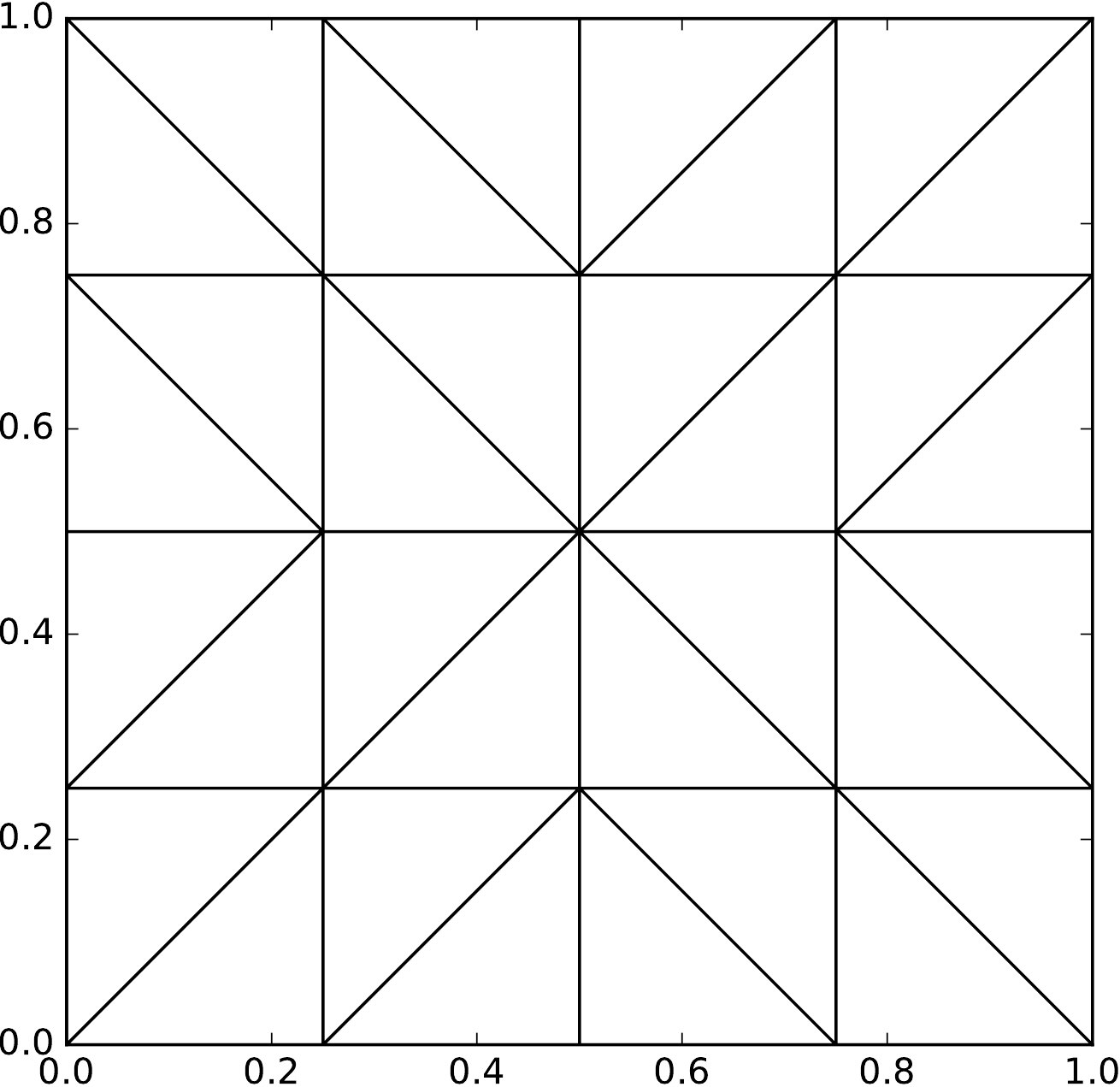}
    \includegraphics[width=0.45\textwidth]{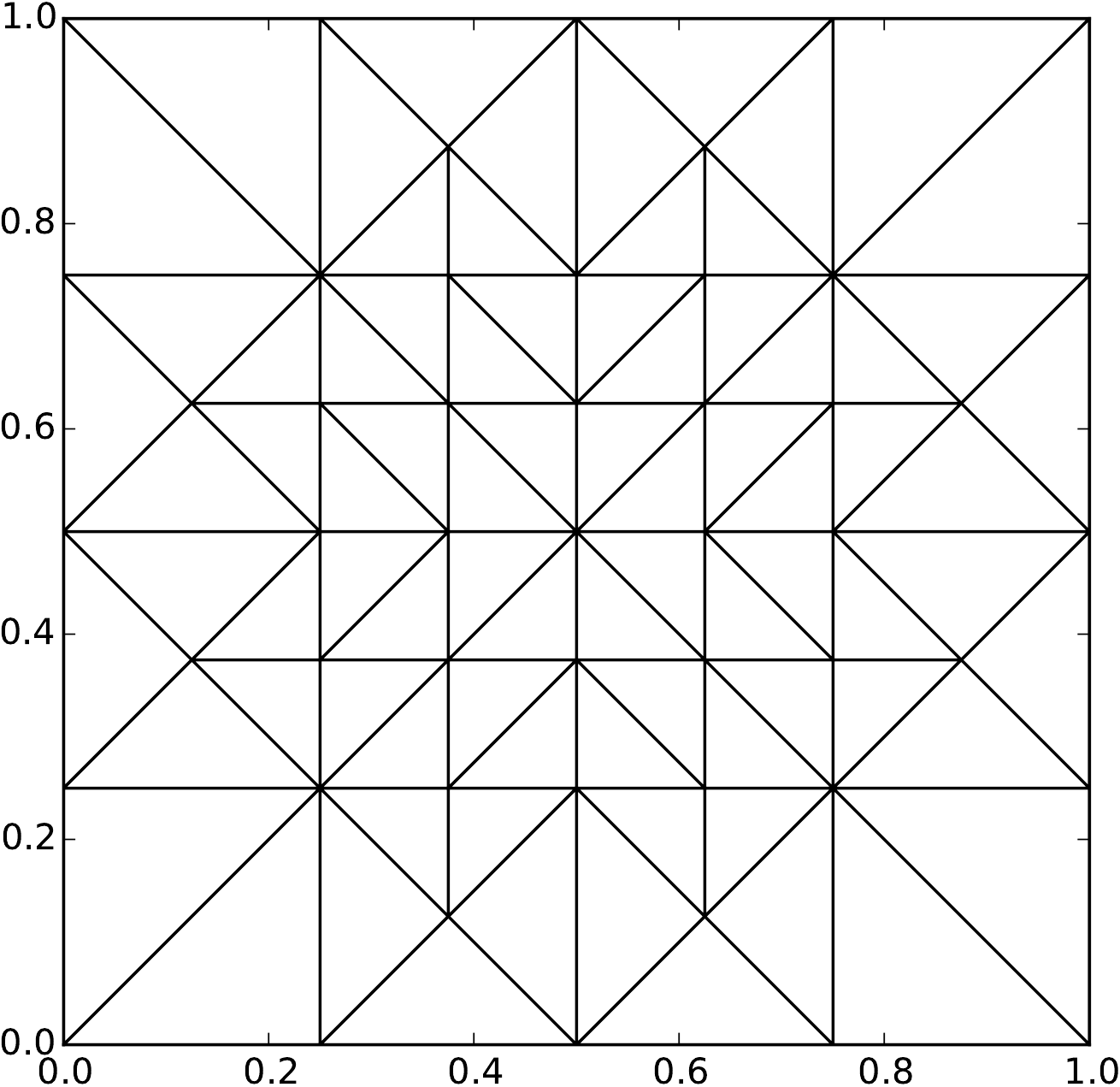}\\[0.5cm]
    \includegraphics[width=0.45\textwidth]{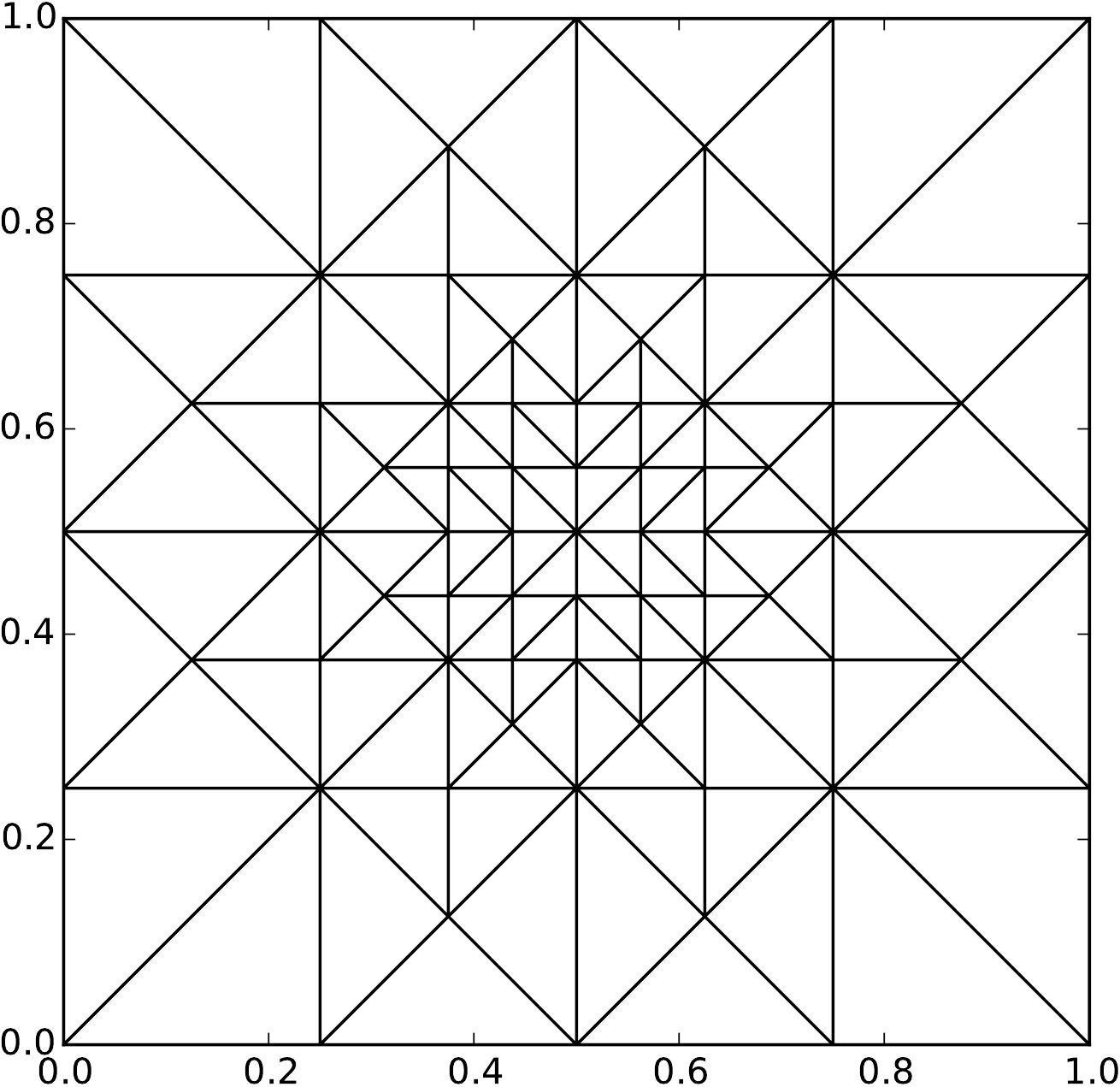}
    \includegraphics[width=0.45\textwidth]{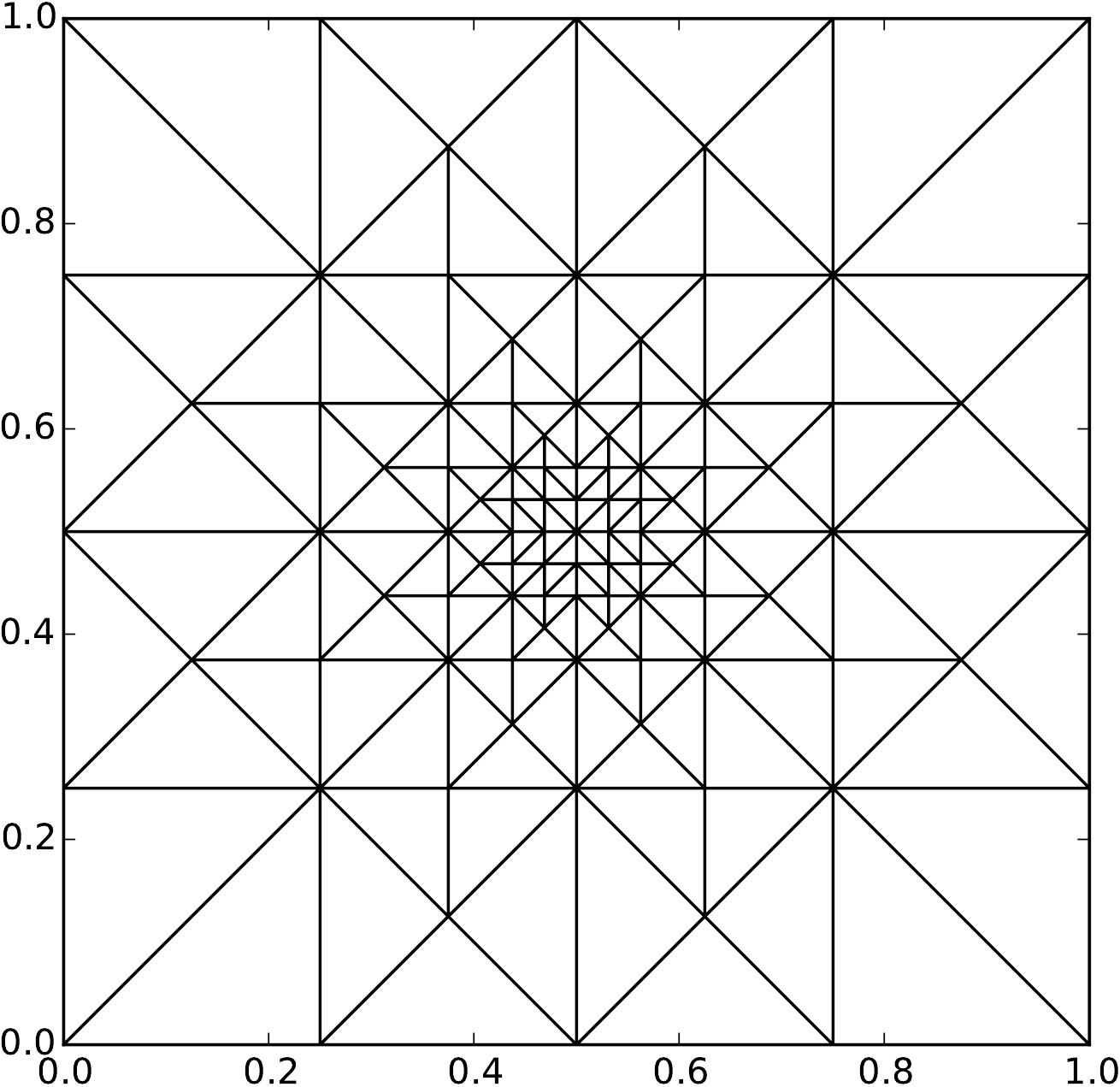}\\[0.5cm]
    \includegraphics[width=0.45\textwidth]{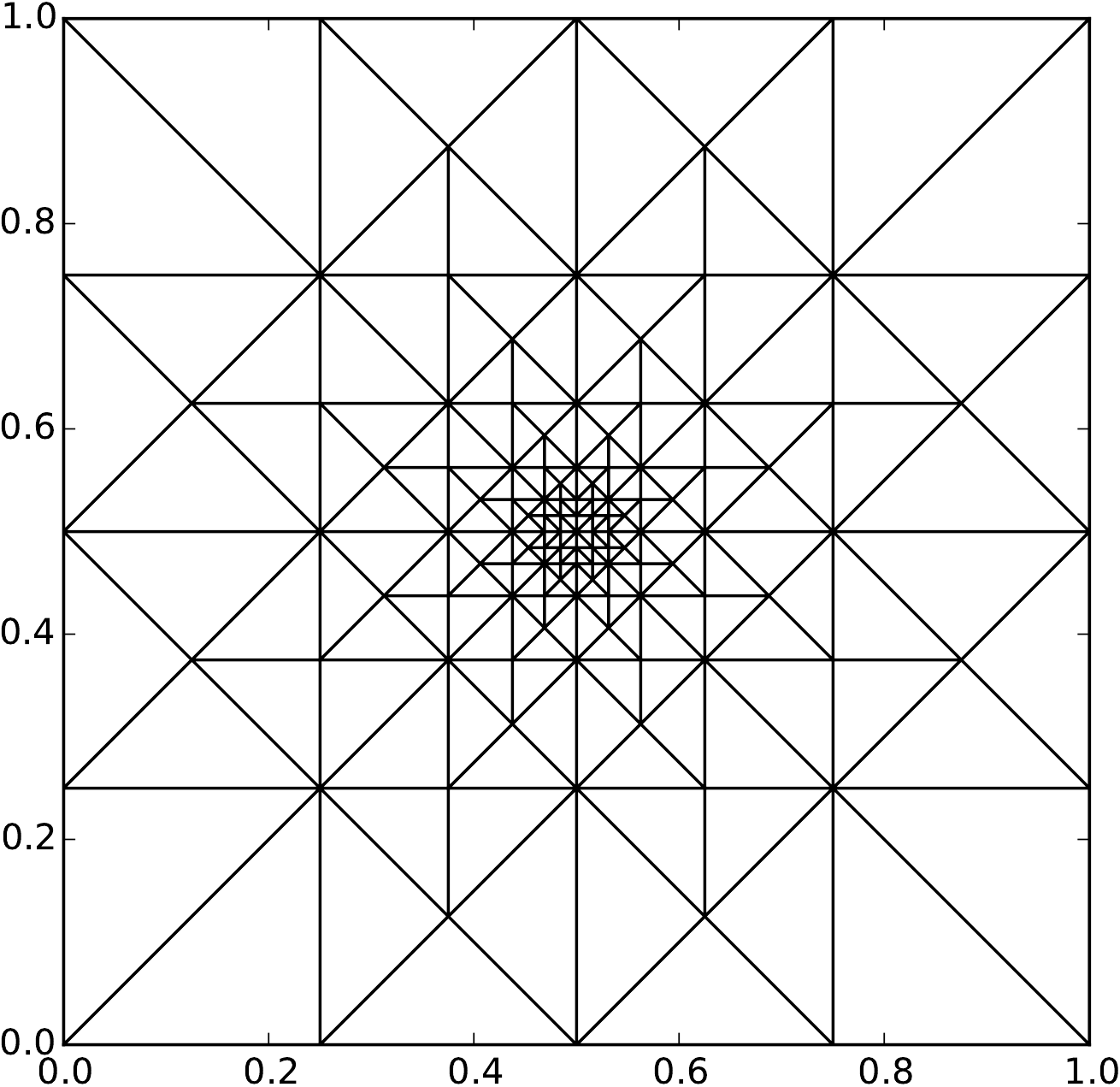}
    \includegraphics[width=0.45\textwidth]{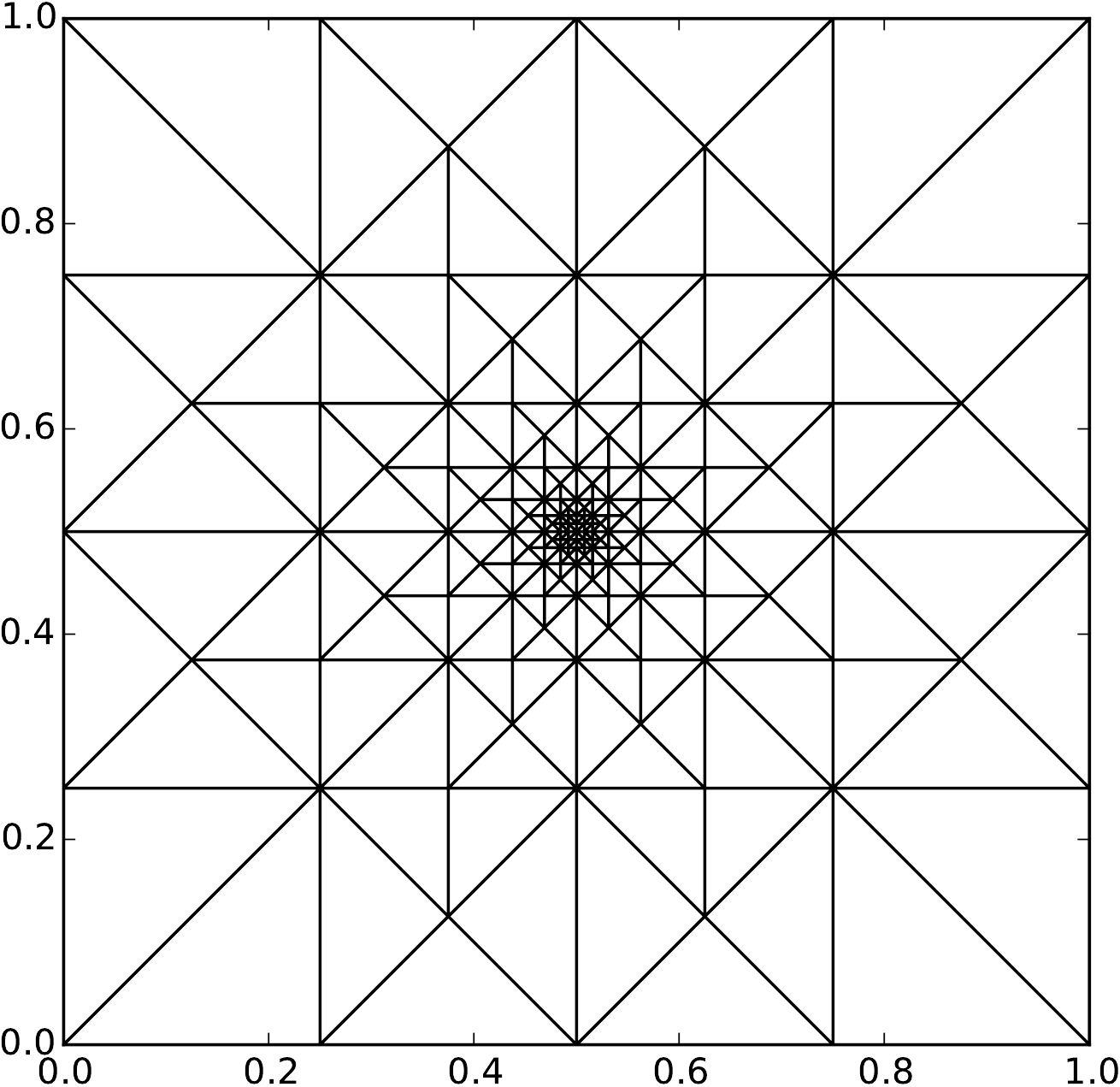}
    \caption{The sequence of adaptive meshes in  the rigid obstacle case.}
    \label{fig:meshrigid}
\end{figure}

\begin{figure}[h!]
    \centering
    \begin{tikzpicture}[scale=0.9]
        \begin{axis}[
                xmode = log,
                ymode = log,
                xlabel = {$N$},
                ylabel = {$\eta + S$},
                grid = both
            ]
            \addplot table[x=ndofs,y=estimator] {\adaptivesquare};
            \addplot table[x=ndofs,y=estimator] {\uniformsquare};
            \addplot+ [black, domain=1000:2000, mark=none, dashed] {exp(-0.5*ln(x) + ln(4) - (-0.5)*ln(1000)))} node[below,pos=1.0]{$\mathcal{O}(N^{-1/2})$};
            \addplot+ [black, domain=400:600, mark=none, dashed] {exp(-2.0*ln(x) + ln(4) - (-2)*ln(400)))} node[below,pos=1.0]{$\mathcal{O}(N^{-2})$};
            \addlegendentry{Adaptive}
            \addlegendentry{Uniform}
        \end{axis}
    \end{tikzpicture}
    \caption{The global error estimator, plotted as a
    function of the number of degrees of freedom $N$, in the rigid obstacle case. The optimal rate of
    convergence for the Argyris element, $\mathcal{O}(N^{-2})$,
    is obtained by the adaptive meshing strategy. The 
   regularity of the exact solution limits the convergence rate in uniform refinement.}
    \label{fig:resultsrigid}
\end{figure}
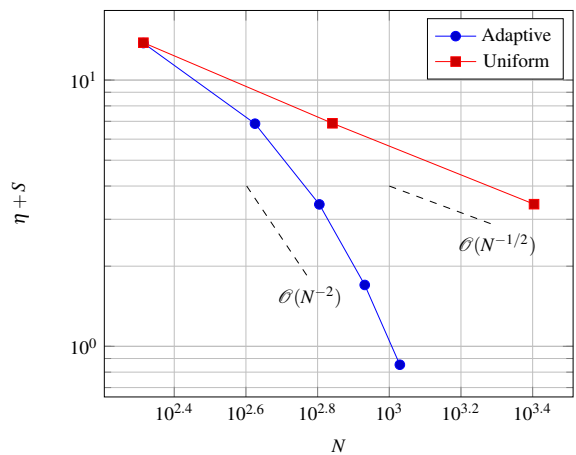

\pgfplotstableread{
    ndofs estimator
    206 2.40958145922
    694 0.604598919111
    2534 0.135130068858
}\uniformelastic

\pgfplotstableread{
    ndofs estimator
    206 0.423701391706
    694 0.0730532010207
    2534 0.0109649559129
}\uniformelasticI

\pgfplotstableread{
    ndofs estimator
    206 0.71200523092
    694 0.132325162517
    2534 0.0224916697562
}\uniformelasticII

\pgfplotstableread{
    ndofs estimator
    206 1.3029015311
    694 0.30351464721
    2534 0.0560139964278
}\uniformelasticIII

\pgfplotstableread{
    ndofs estimator
    206 1.3029015311
    422 0.393588486774
    1026 0.0964961331776
    1918 0.0337280568804
    2414 0.0203236456007
}\adaptiveeps

\pgfplotstableread{
    ndofs estimator
    206 0.712005230923
    422 0.282751982066
    594 0.132221293127
    1486 0.0270084591846
    2558 0.00959790612026
}\adaptiveepsI

\pgfplotstableread{
    ndofs estimator
    206 2.40958145922
    422 0.664886900657
    854 0.1933154288
    1782 0.0598013416591
    2070 0.0473621472658
}\adaptiveepsII

\pgfplotstableread{
    ndofs estimator
    206 0.423701391706
    594 0.0743198859836
    1378 0.0210204451689
    1946 0.0107098177271
    2874 0.00443523214802
}\adaptiveepsIII

\begin{figure}[h!]
    \centering
    \begin{tikzpicture}[scale=0.9]
        \begin{axis}[
                xmode = log,
                ymode = log,
                xlabel = {$N$},
                ylabel = {$\eta + S$},
                grid = both,
                legend pos=outer north east
            ]
            \addplot table[x=ndofs,y=estimator] {\uniformelastic};
            \addplot table[x=ndofs,y=estimator] {\uniformelasticIII};
            \addplot table[x=ndofs,y=estimator] {\uniformelasticII};
            \addplot table[x=ndofs,y=estimator] {\uniformelasticI};
            \addplot+ [black, domain=1000:2000, mark=none, dashed] {exp(-1.0*ln(x) + ln(0.6) - (-1)*ln(1000)))} node[above,pos=0.0]{$\mathcal{O}(N^{-1})$};
            \addlegendentry{Uniform, $\epsilon = 10^{-6}$}
            \addlegendentry{Uniform, $\epsilon = 10^{-5}$}
            \addlegendentry{Uniform, $\epsilon = 10^{-4}$}
            \addlegendentry{Uniform, $\epsilon = 10^{-3}$}
        \end{axis}
    \end{tikzpicture}\\
    \begin{tikzpicture}[scale=0.9]
        \begin{axis}[
                xmode = log,
                ymode = log,
                xlabel = {$N$},
                ylabel = {$\eta + S$},
                grid = both,
                legend pos=outer north east
            ]
            \addplot table[x=ndofs,y=estimator] {\adaptiveepsII};
            \addplot table[x=ndofs,y=estimator] {\adaptiveeps};
            \addplot table[x=ndofs,y=estimator] {\adaptiveepsI};
            \addplot table[x=ndofs,y=estimator] {\adaptiveepsIII};
            \addplot+ [black, domain=1000:2000, mark=none, dashed] {exp(-2.0*ln(x) + ln(2e-2) - (-2)*ln(1000)))} node[below,pos=1.0]{$\mathcal{O}(N^{-2})$};
            \addlegendentry{Adaptive, $\epsilon = 10^{-6}$}
            \addlegendentry{Adaptive, $\epsilon = 10^{-5}$}
            \addlegendentry{Adaptive, $\epsilon = 10^{-4}$}
            \addlegendentry{Adaptive, $\epsilon = 10^{-3}$}
        \end{axis}
    \end{tikzpicture}\\
    \caption{The global error estimator in the elastic case plotted as a function of the number of
    degrees of freedom $N$. The upper and lower diagrams  correspond to the uniform
    and the adaptive refinements, respectively.}
    \label{fig:elastic}
\end{figure}
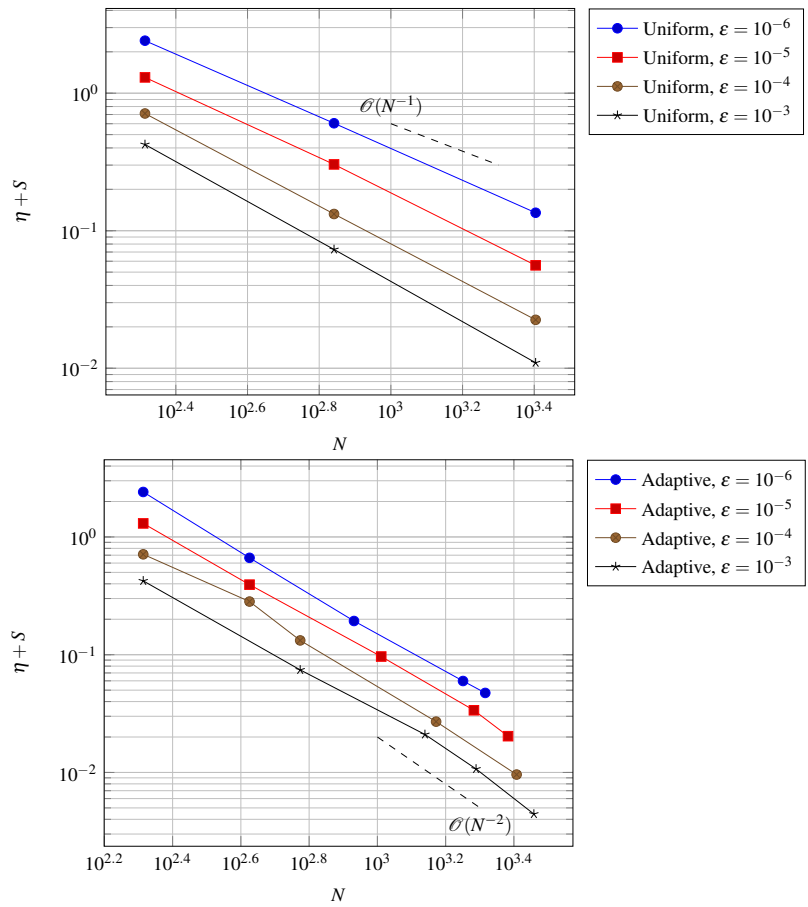

\begin{figure}
    \centering
    \begin{subfigure}{0.49\textwidth}
        \includegraphics[width=\textwidth,trim={1.5cm 0.5cm 1.5cm 0.5cm},clip]{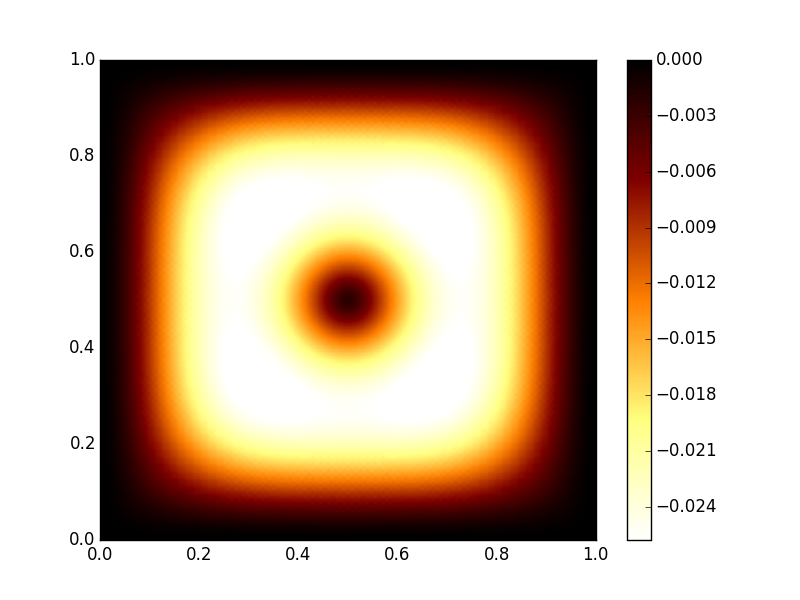}
        \caption{$\epsilon = 0$}
    \end{subfigure}\\
    \begin{subfigure}{0.49\textwidth}
        \includegraphics[width=\textwidth,trim={1.5cm 0.5cm 1.5cm 0.5cm},clip]{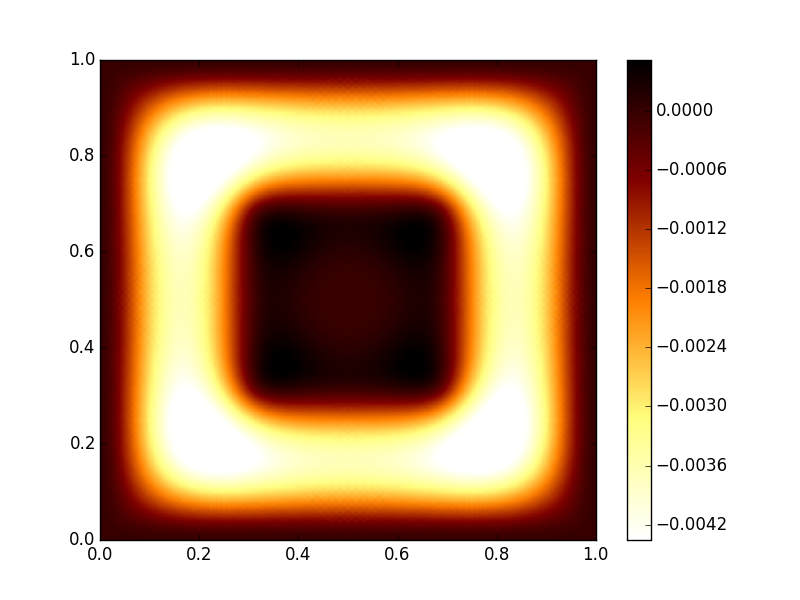}
        \caption{$\epsilon = 10^{-6}$}
    \end{subfigure}
    \begin{subfigure}{0.49\textwidth}
        \includegraphics[width=\textwidth,trim={1.5cm 0.5cm 1.5cm 0.5cm},clip]{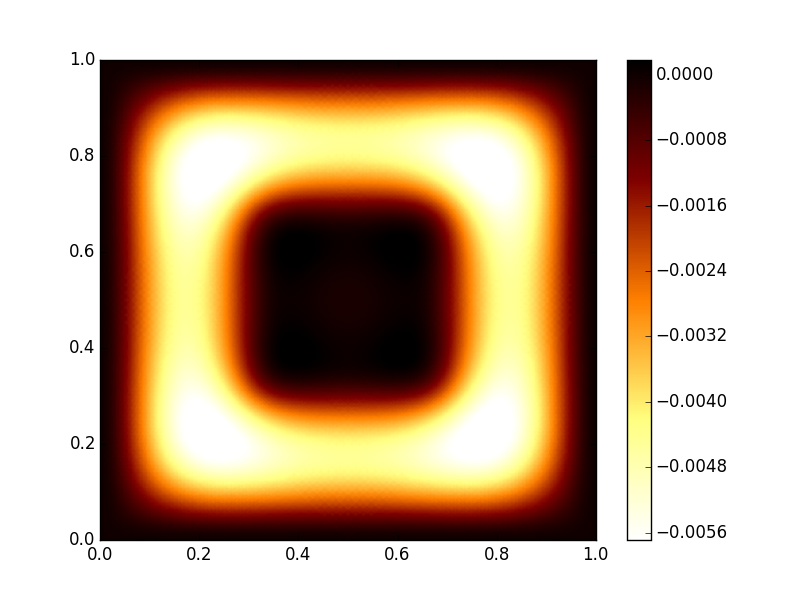}
        \caption{$\epsilon = 10^{-5}$}
    \end{subfigure}
    \begin{subfigure}{0.49\textwidth}
        \includegraphics[width=\textwidth,trim={1.5cm 0.5cm 1.5cm 0.5cm},clip]{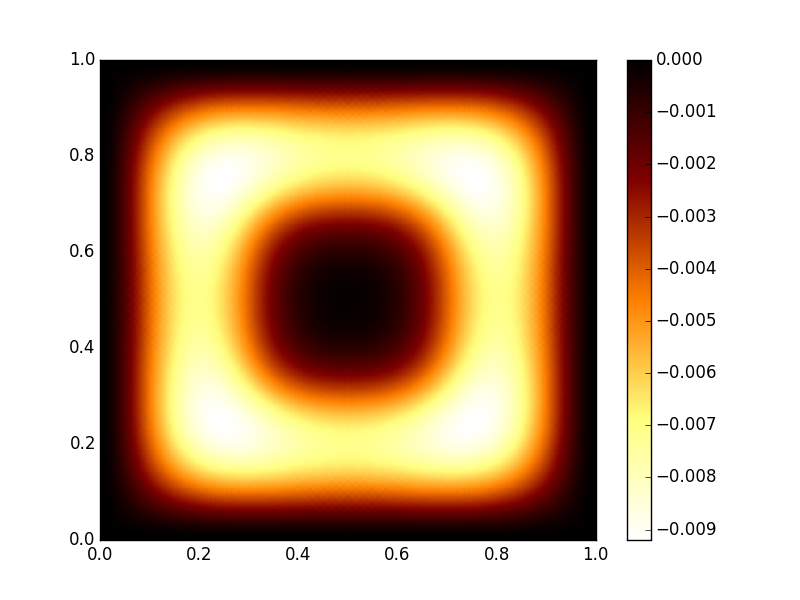}
        \caption{$\epsilon = 10^{-4}$}
    \end{subfigure}
    \begin{subfigure}{0.49\textwidth}
        \includegraphics[width=\textwidth,trim={1.5cm 0.5cm 1.5cm 0.5cm},clip]{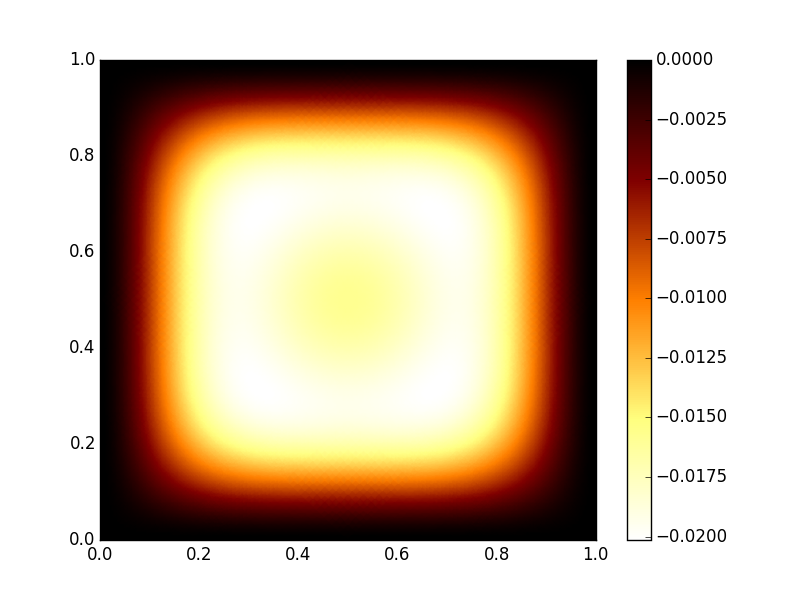}
        \caption{$\epsilon = 10^{-3}$}
    \end{subfigure}
    \caption{The discrete displacements shown for five different values of $\epsilon$ after three adaptive refinements in each  case.
    Note that the solutions are visualised on a more refined mesh.}
    \label{fig:discrete}
\end{figure}

\begin{figure}
    \centering
    \begin{subfigure}{0.49\textwidth}
        \includegraphics[width=\textwidth,trim={1.5cm 0.5cm 1.5cm 0.5cm},clip]{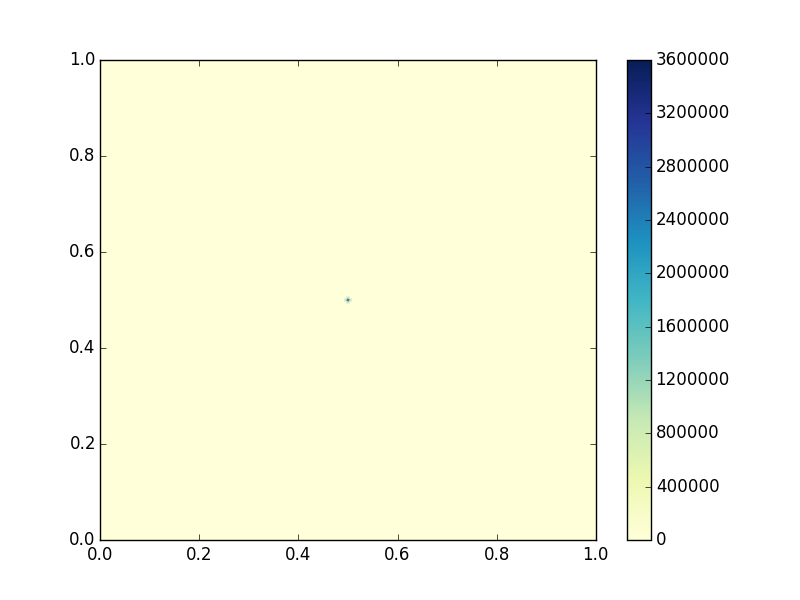}
        \caption{$\epsilon = 0$}
    \end{subfigure}\\
    \begin{subfigure}{0.49\textwidth}
        \includegraphics[width=\textwidth,trim={1.5cm 0.5cm 1.5cm 0.5cm},clip]{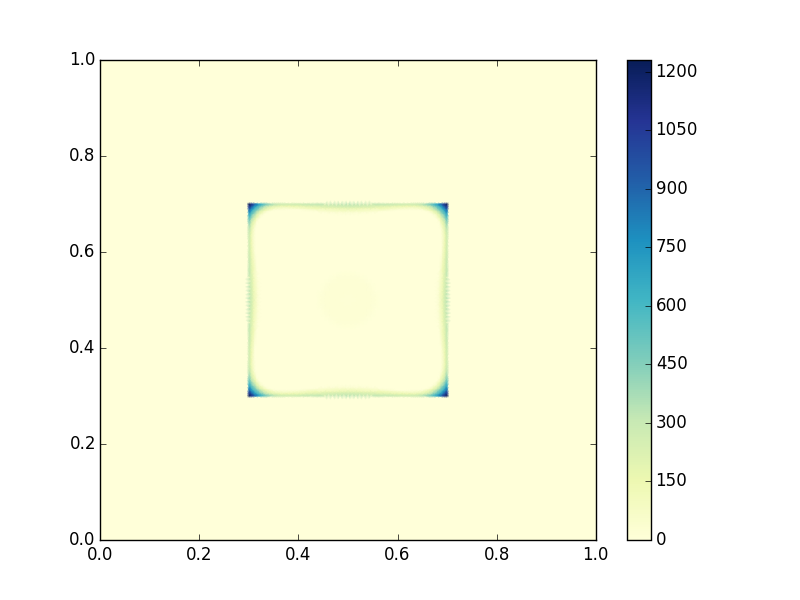}
        \caption{$\epsilon = 10^{-6}$}
    \end{subfigure}
    \begin{subfigure}{0.49\textwidth}
        \includegraphics[width=\textwidth,trim={1.5cm 0.5cm 1.5cm 0.5cm},clip]{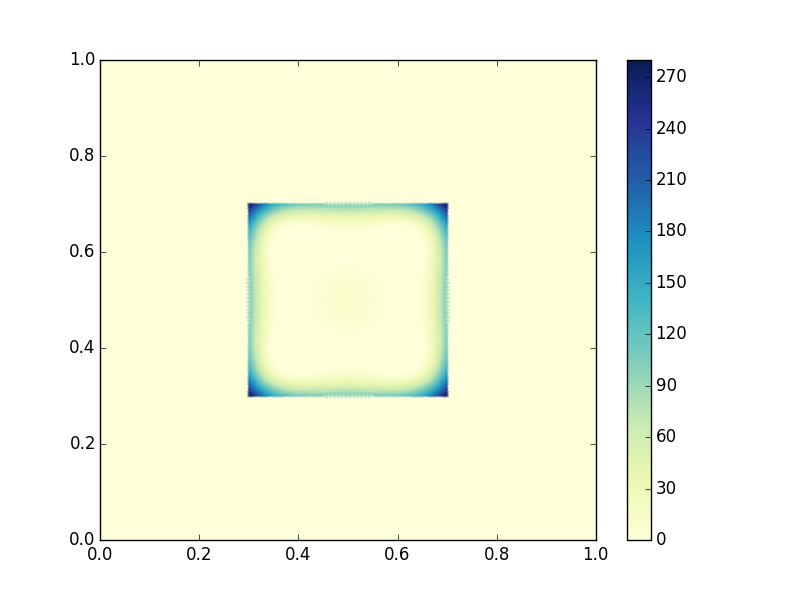}
        \caption{$\epsilon = 10^{-5}$}
    \end{subfigure}
    \begin{subfigure}{0.49\textwidth}
        \includegraphics[width=\textwidth,trim={1.5cm 0.5cm 1.5cm 0.5cm},clip]{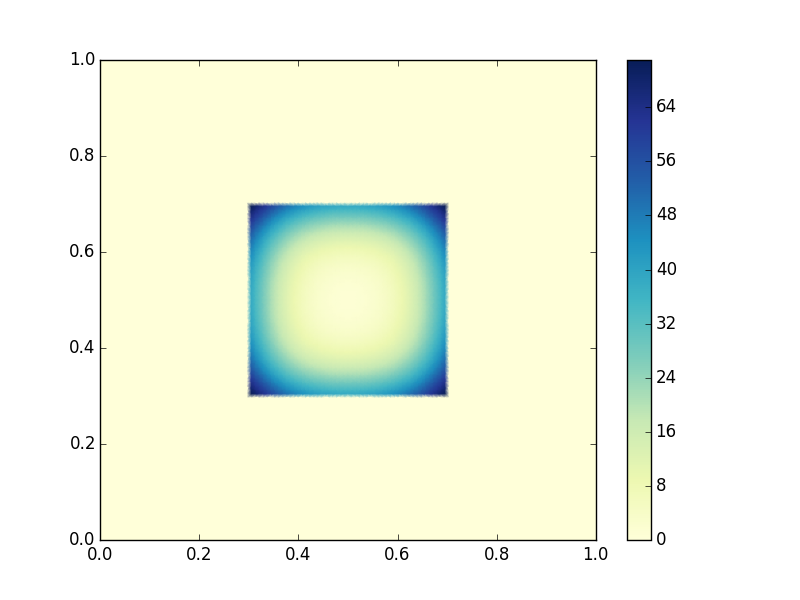}
        \caption{$\epsilon = 10^{-4}$}
    \end{subfigure}
    \begin{subfigure}{0.49\textwidth}
        \includegraphics[width=\textwidth,trim={1.5cm 0.5cm 1.5cm 0.5cm},clip]{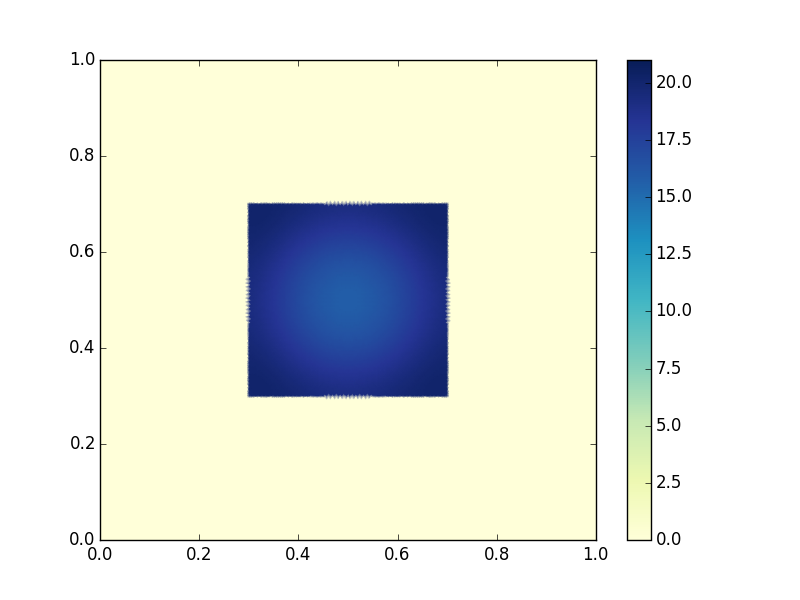}
        \caption{$\epsilon = 10^{-3}$}
    \end{subfigure}
    \caption{The discrete Lagrange multipliers shown for five different values of $\epsilon$  after three adaptive refinements in each case.
    Note that the solutions are visualised on a more refined mesh.}
    \label{fig:lagmult}
\end{figure}

\begin{figure}
    \begin{subfigure}{0.49\textwidth}
        \includegraphics[width=\textwidth,trim={2.5cm 0.5cm 2.5cm 0.5cm},clip]{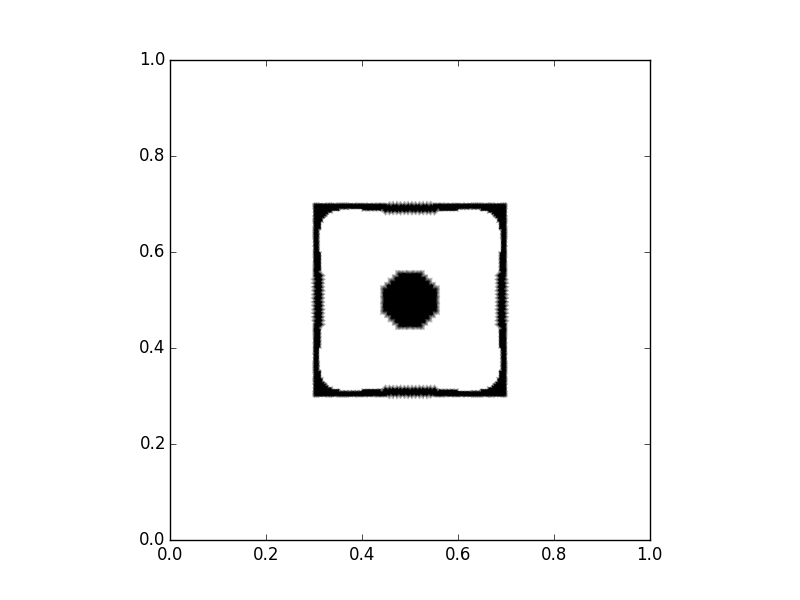}
        \caption{$\epsilon = 10^{-6}$}
    \end{subfigure}
    \begin{subfigure}{0.49\textwidth}
        \includegraphics[width=\textwidth,trim={2.5cm 0.5cm 2.5cm 0.5cm},clip]{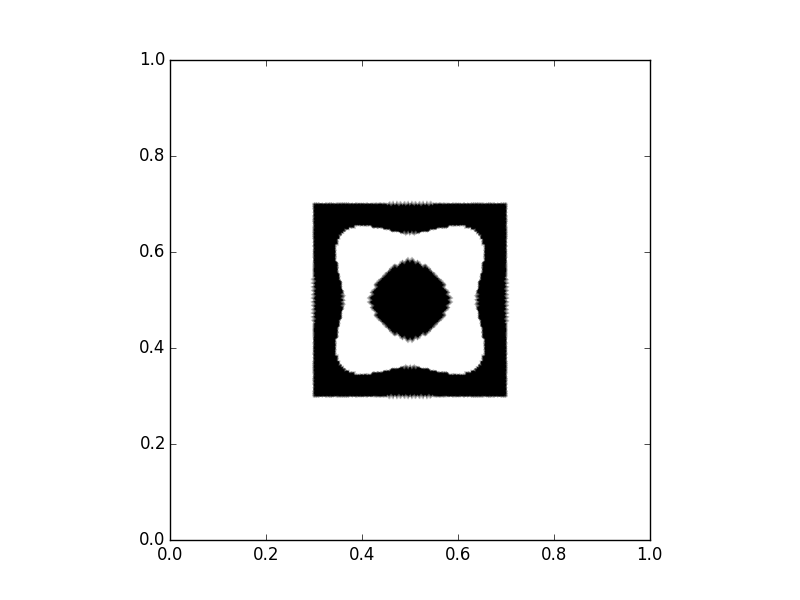}
        \caption{$\epsilon = 10^{-5}$}
    \end{subfigure}
    \begin{subfigure}{0.49\textwidth}
        \includegraphics[width=\textwidth,trim={2.5cm 0.5cm 2.5cm 0.5cm},clip]{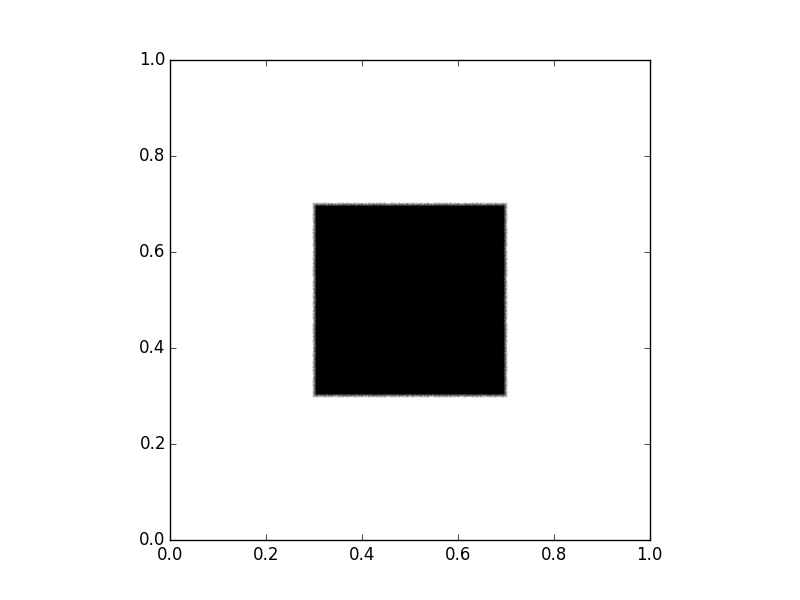}
        \caption{$\epsilon = 10^{-4}$}
    \end{subfigure}
    \begin{subfigure}{0.49\textwidth}
        \includegraphics[width=\textwidth,trim={2.5cm 0.5cm 2.5cm 0.5cm},clip]{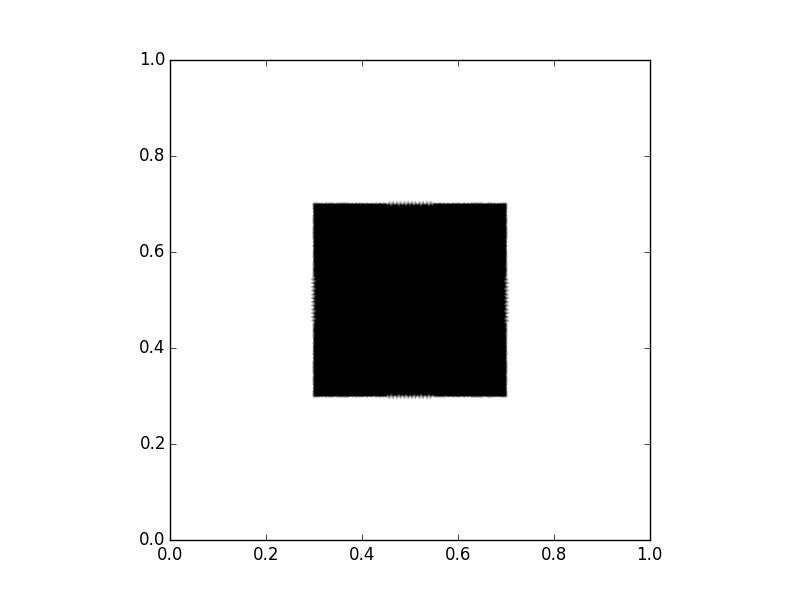}
        \caption{$\epsilon = 10^{-3}$}
    \end{subfigure}
    \caption{The approximate contact sets (i.e.~the regions where the discrete Lagrange multipliers are positive)
    after three adaptive refinements in each of the cases.}
    \label{fig:contactset}
\end{figure}

\begin{figure}
    \includegraphics[width=0.3\textwidth]{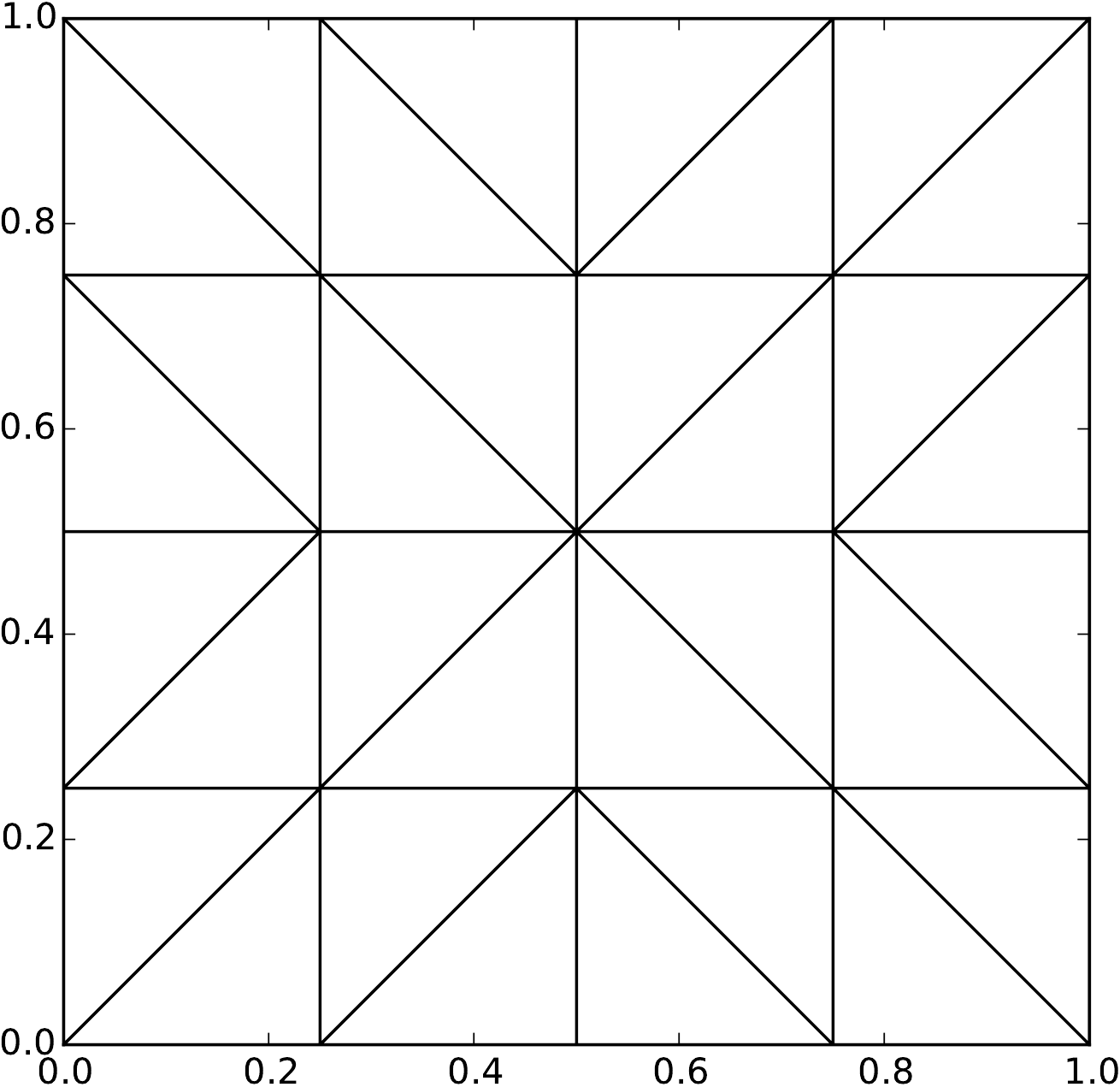}
    \includegraphics[width=0.3\textwidth]{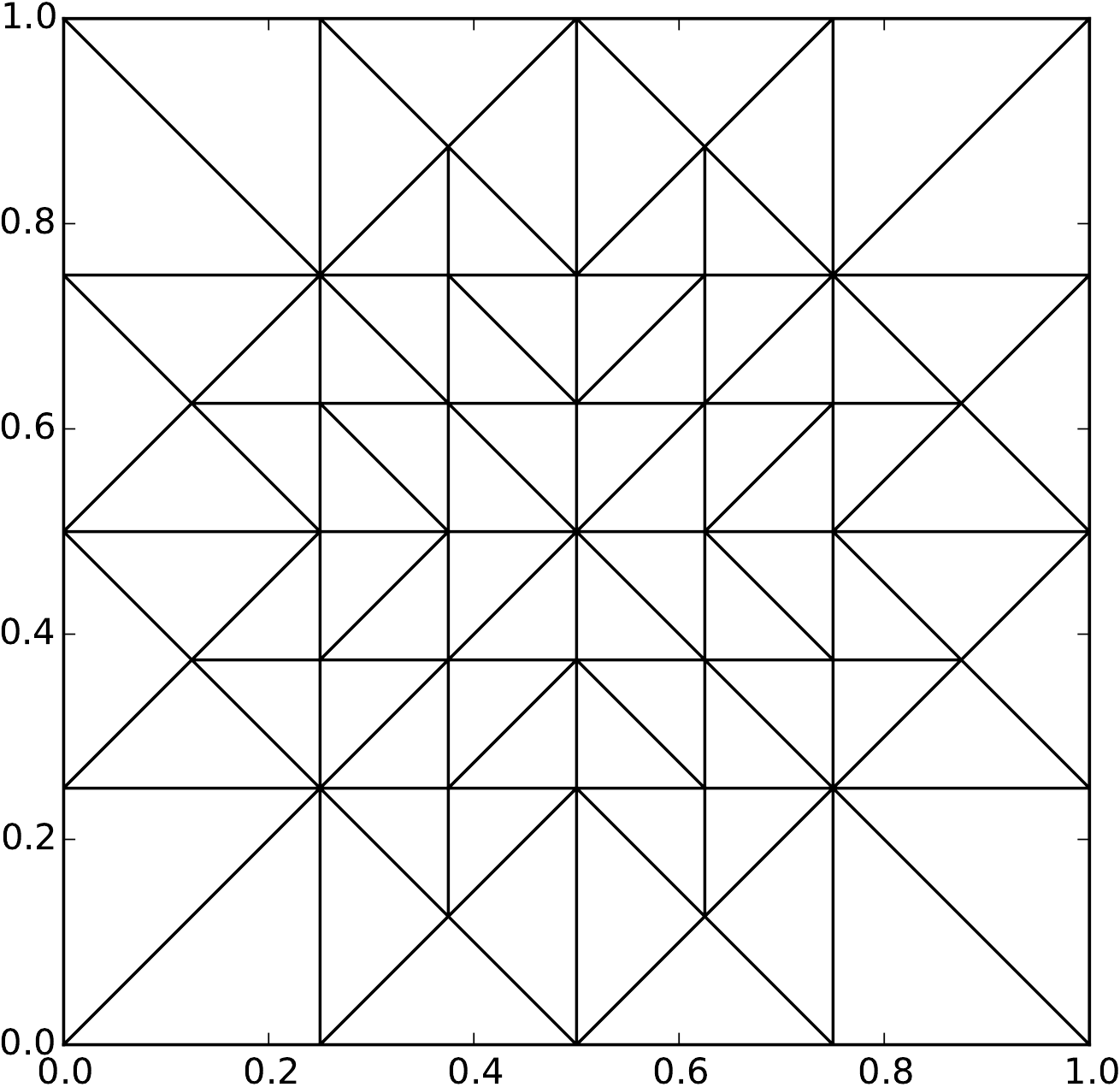}
    \includegraphics[width=0.3\textwidth]{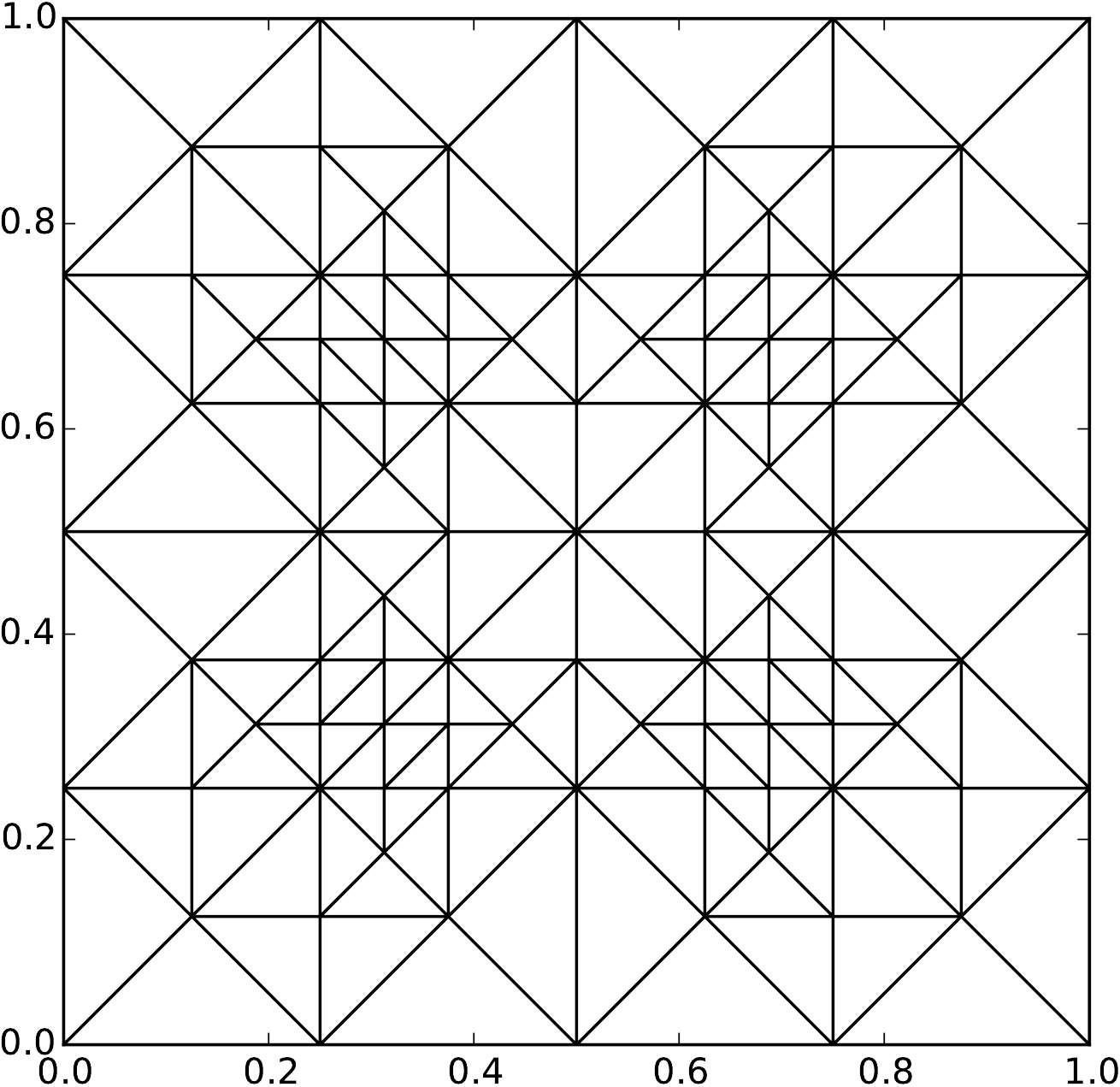}
    \includegraphics[width=0.3\textwidth]{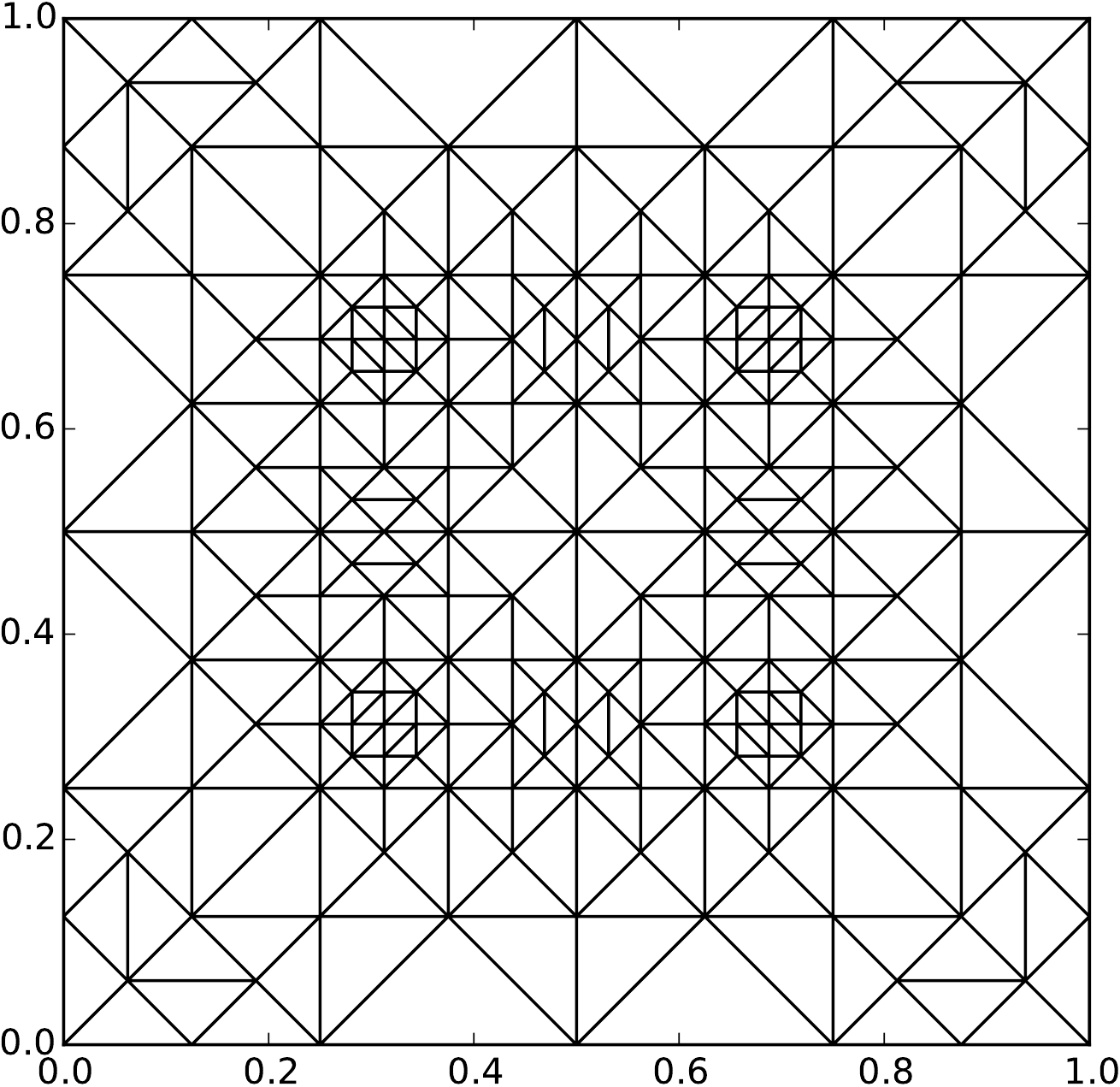}
            \quad  \ \ \ 
    \includegraphics[width=0.3\textwidth]{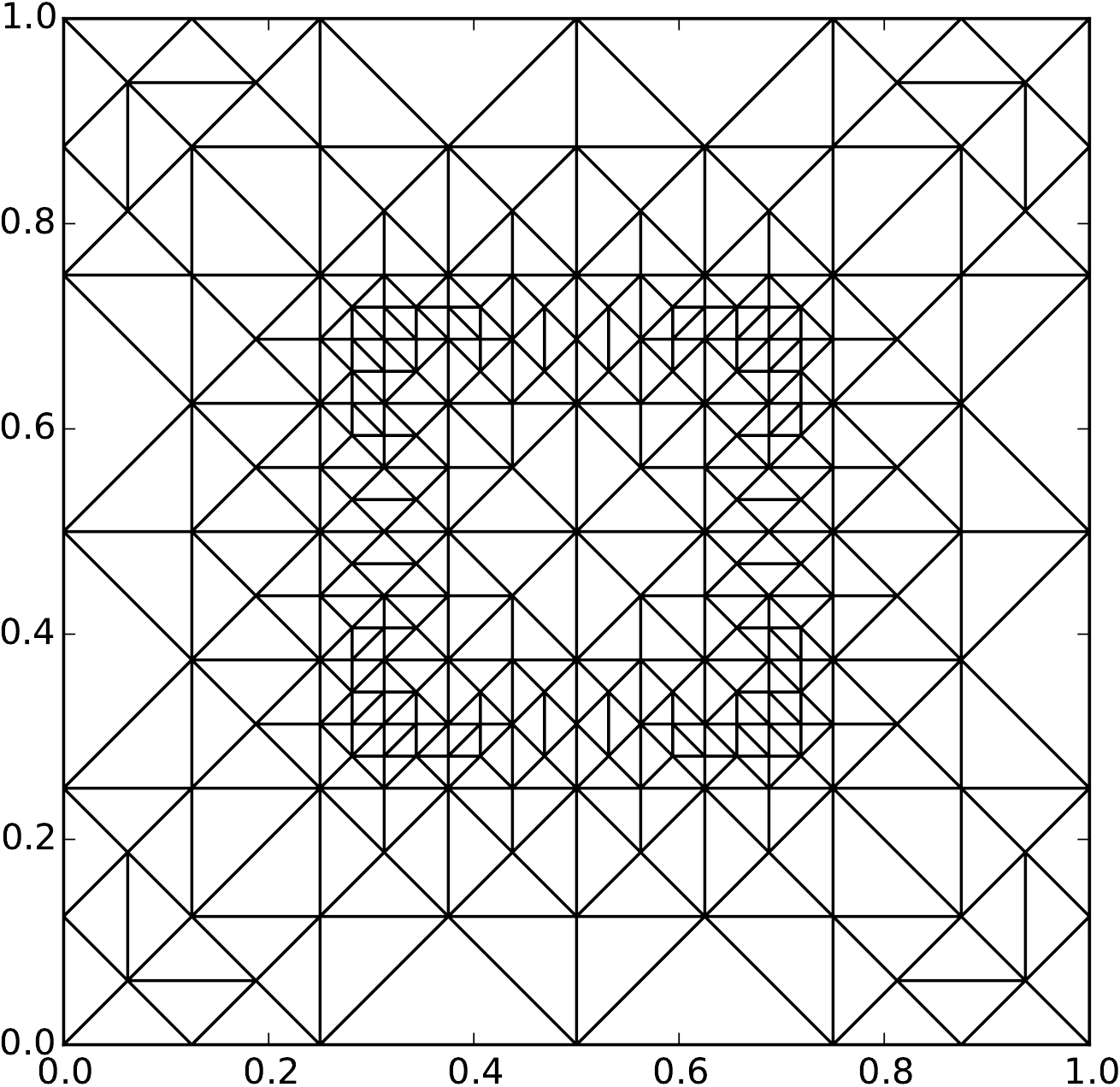}
            \quad  \ \ \ 
    \includegraphics[width=0.3\textwidth]{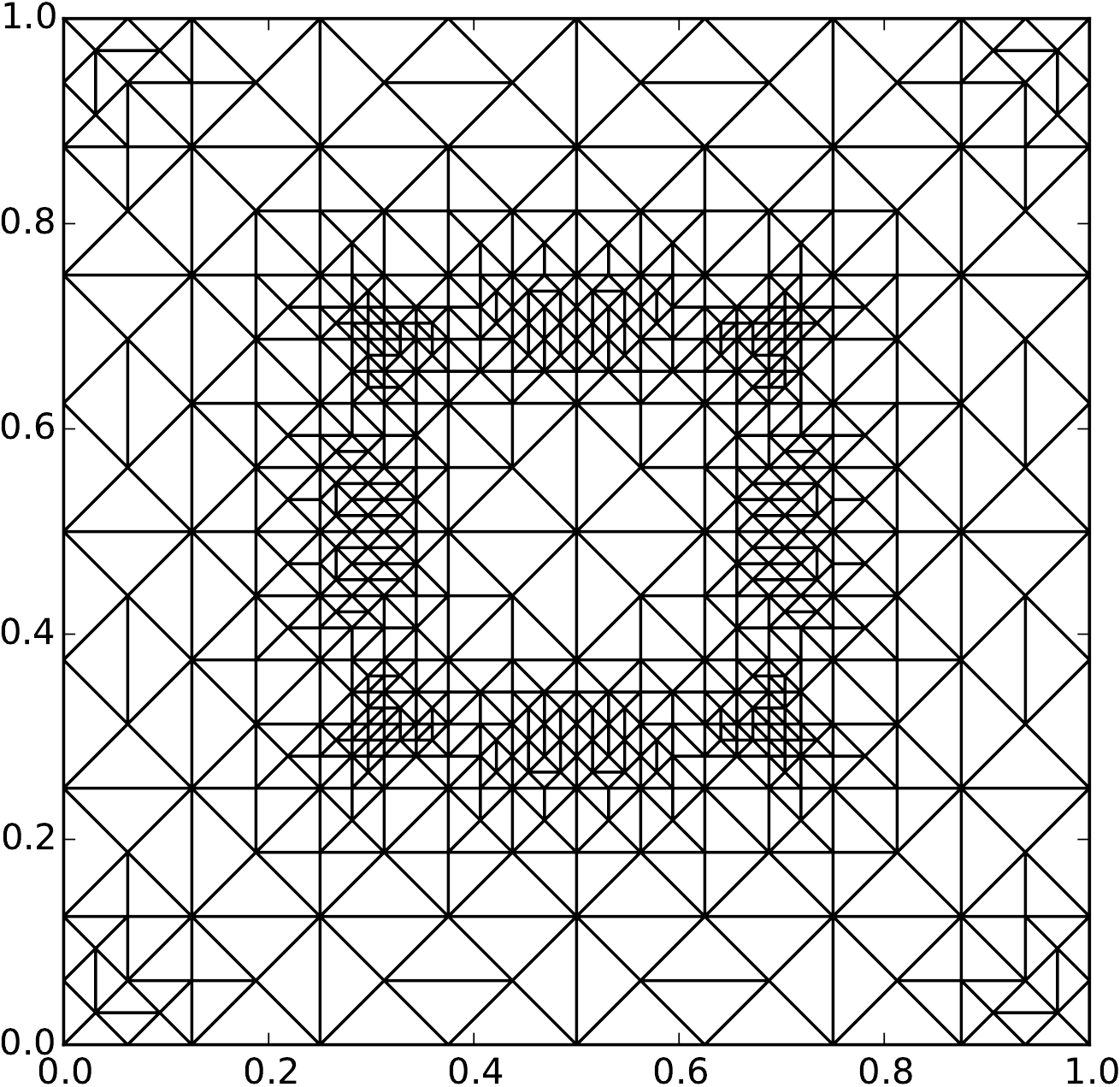}
    \caption{The sequence of adaptive meshes in  the elastic obstacle case with $\epsilon = 10^{-6}$.}
    \label{fig:meshe6}
\end{figure}
\begin{figure}
    \includegraphics[width=0.3\textwidth]{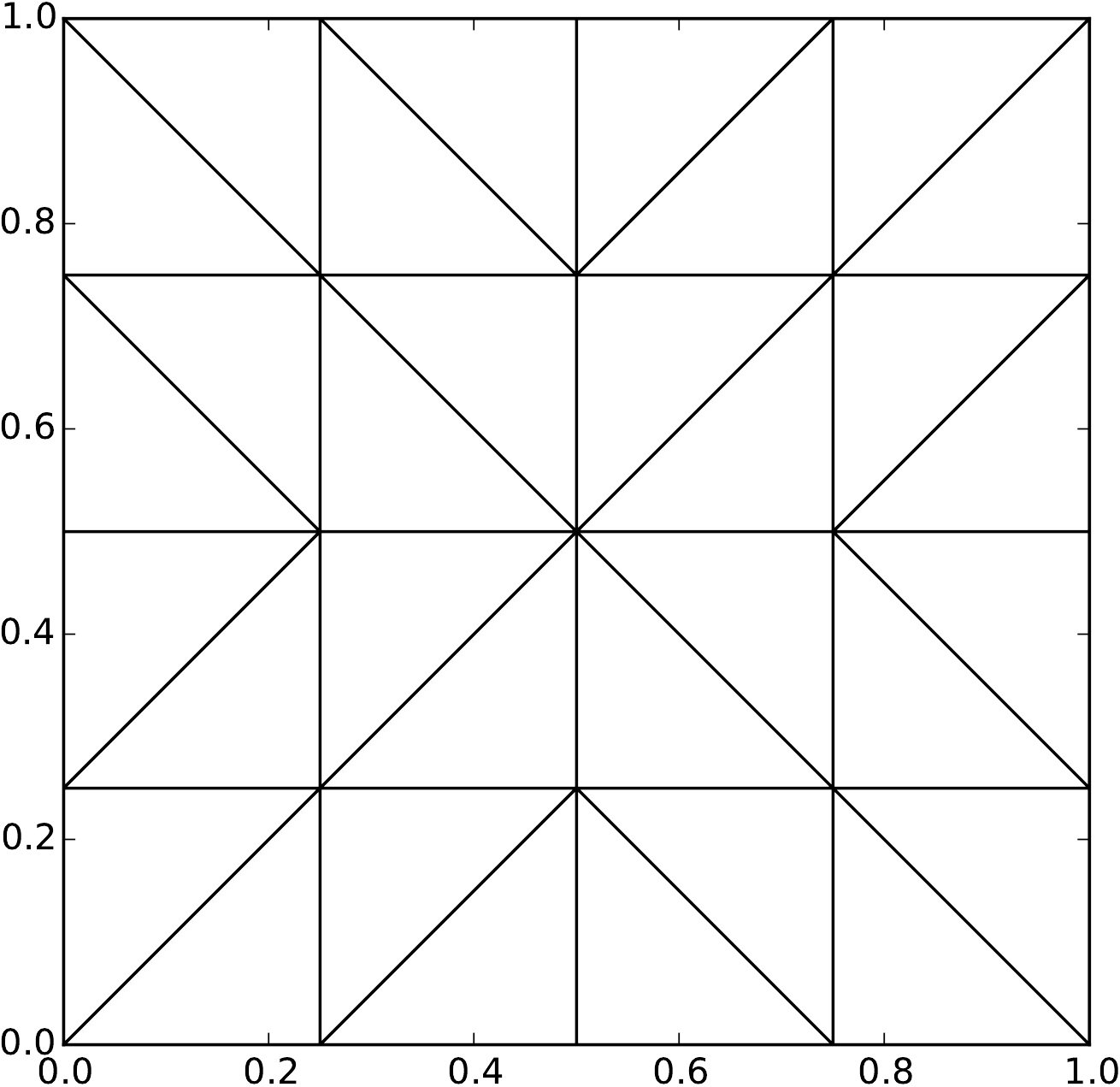}
    \includegraphics[width=0.3\textwidth]{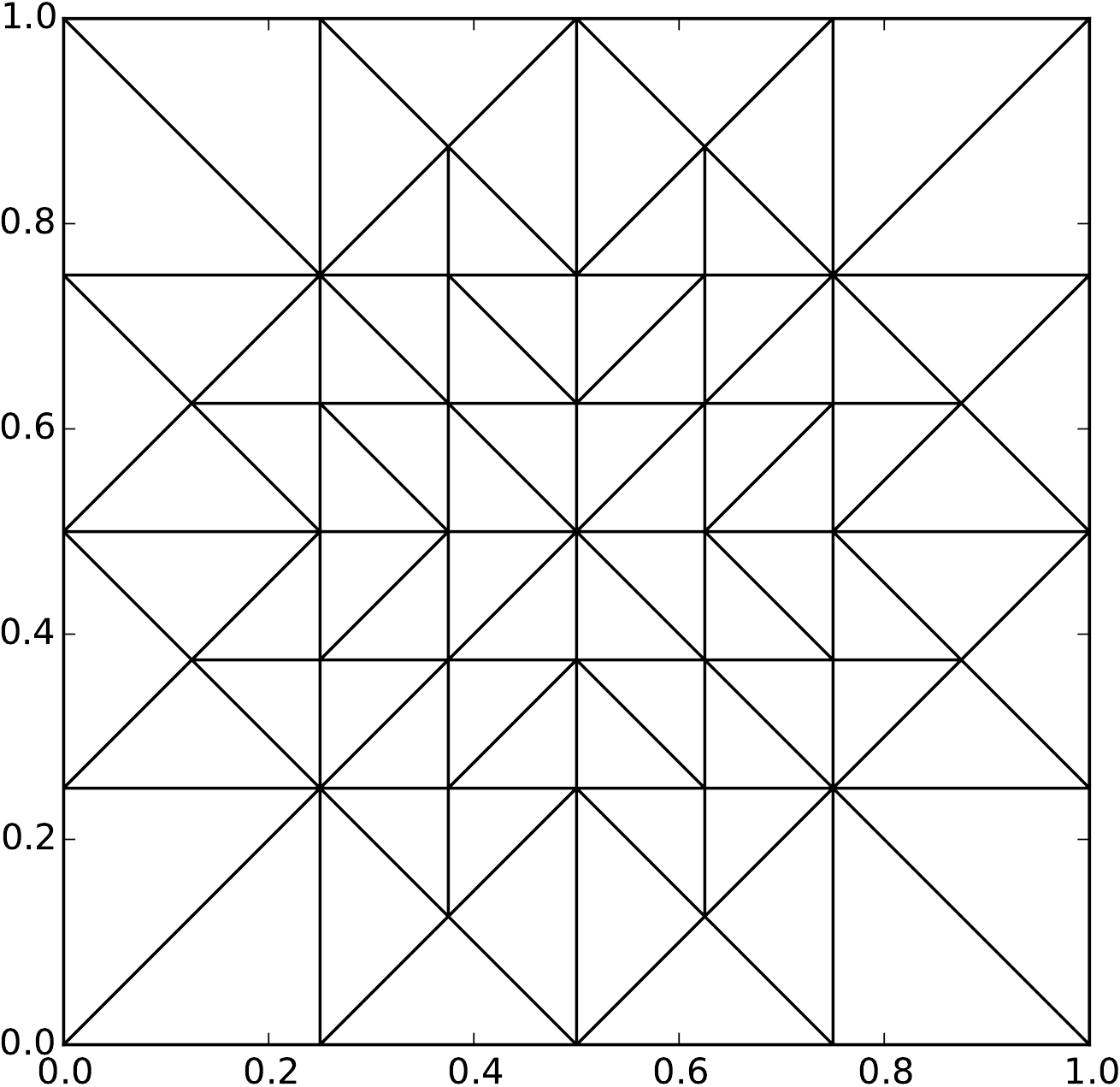}
    \includegraphics[width=0.3\textwidth]{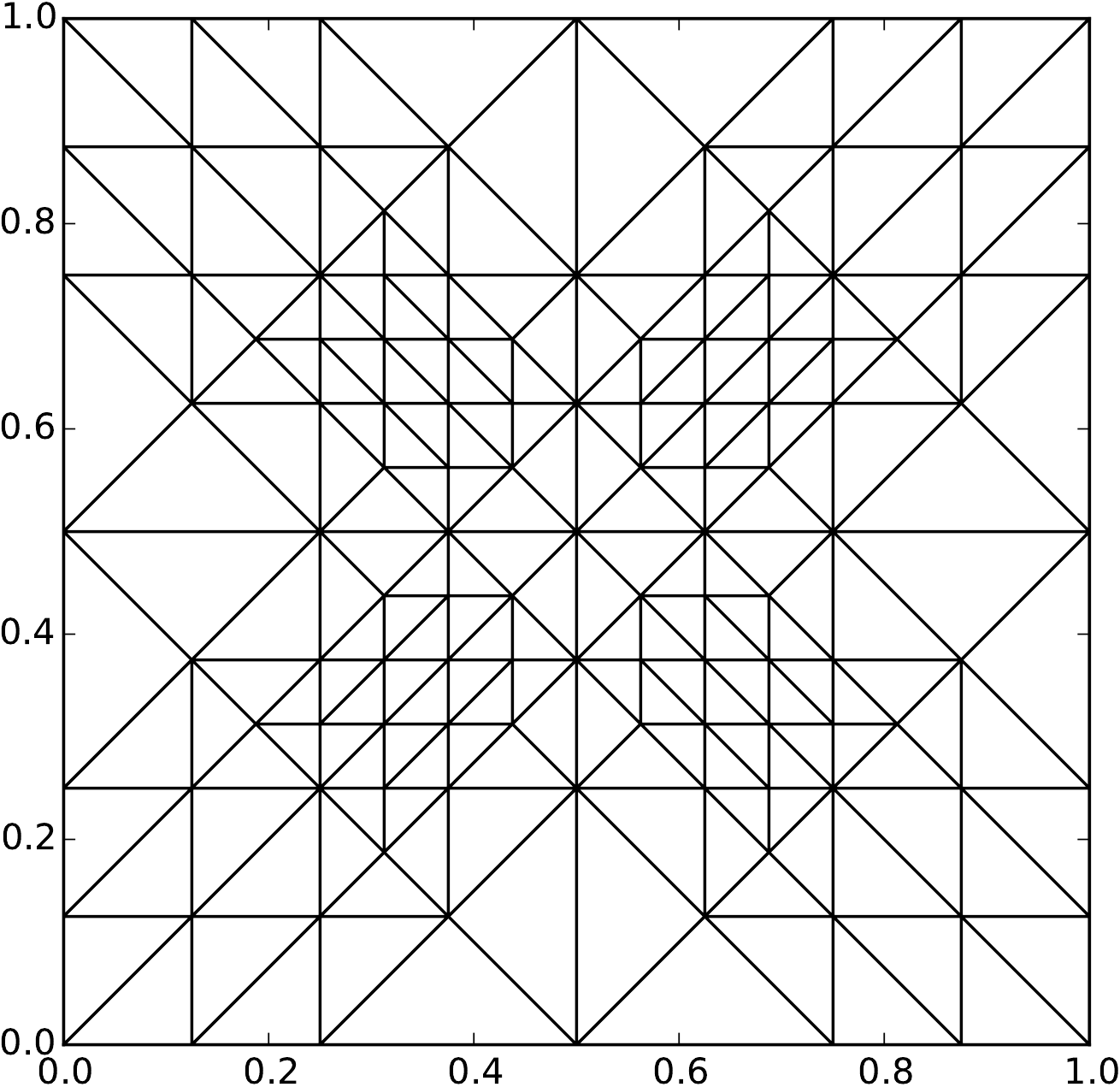}
    \includegraphics[width=0.3\textwidth]{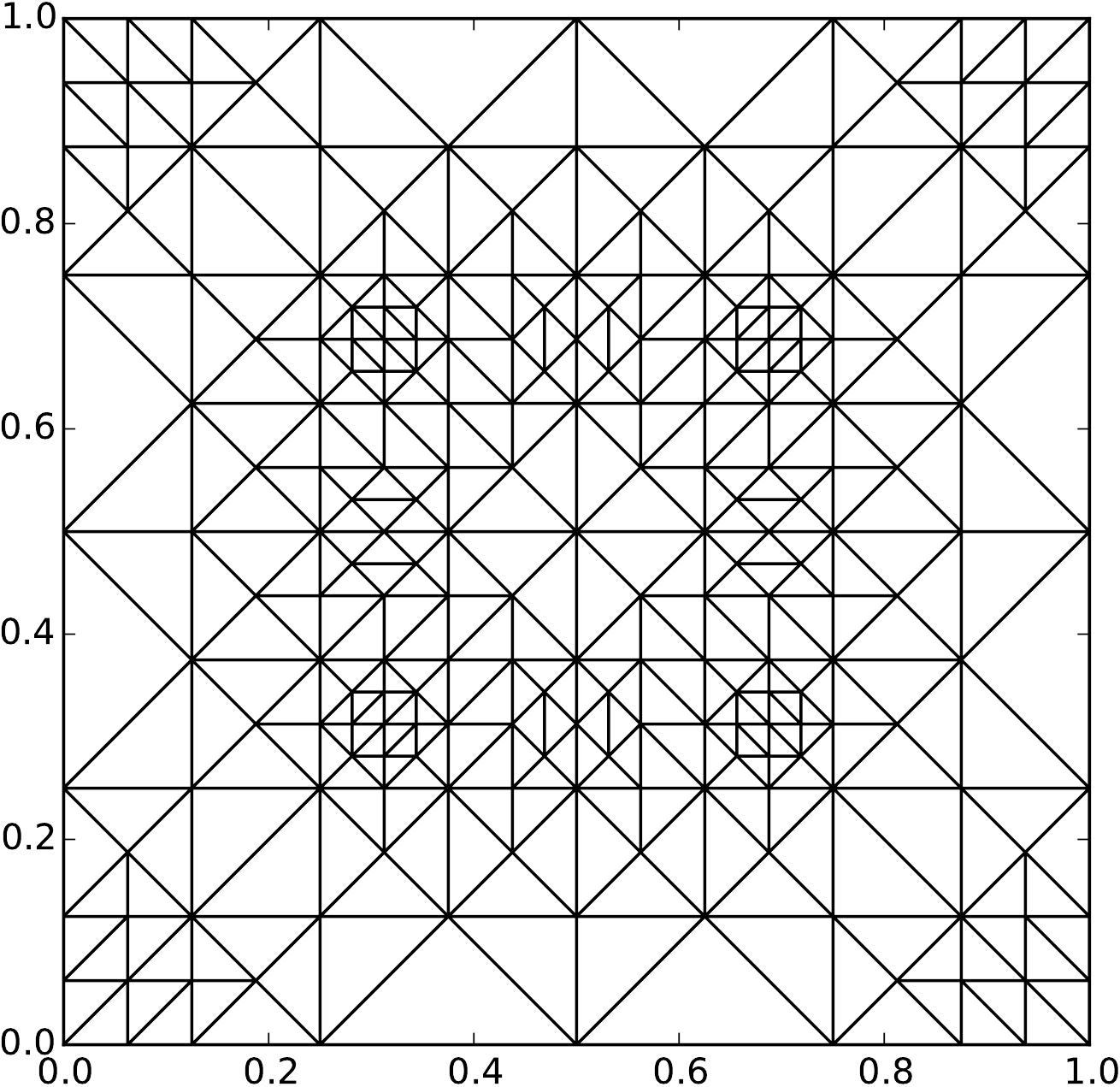}
    \quad  \ \  \ 
    \includegraphics[width=0.3\textwidth]{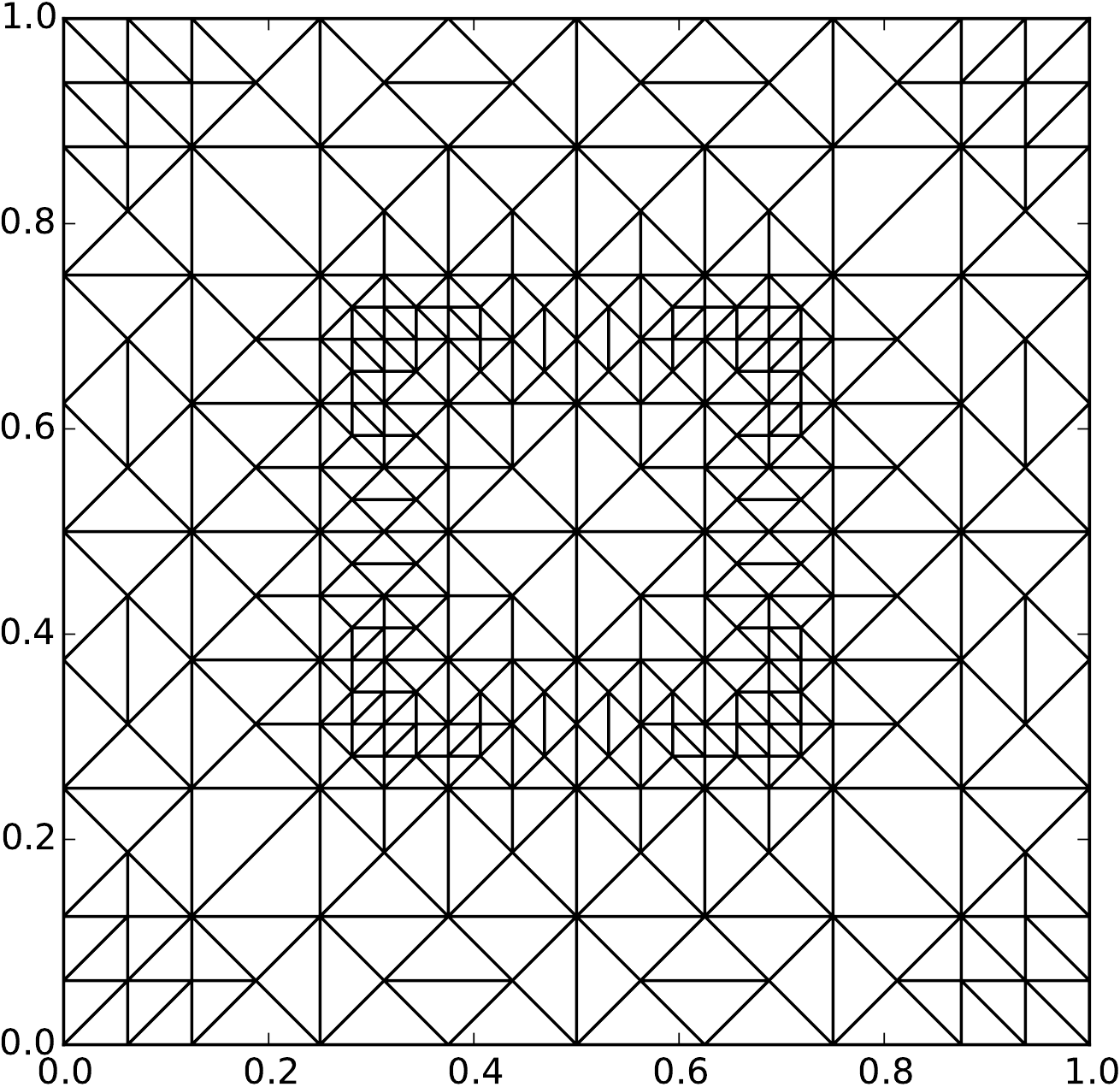}
        \quad  \ \ \ 
    \includegraphics[width=0.3\textwidth]{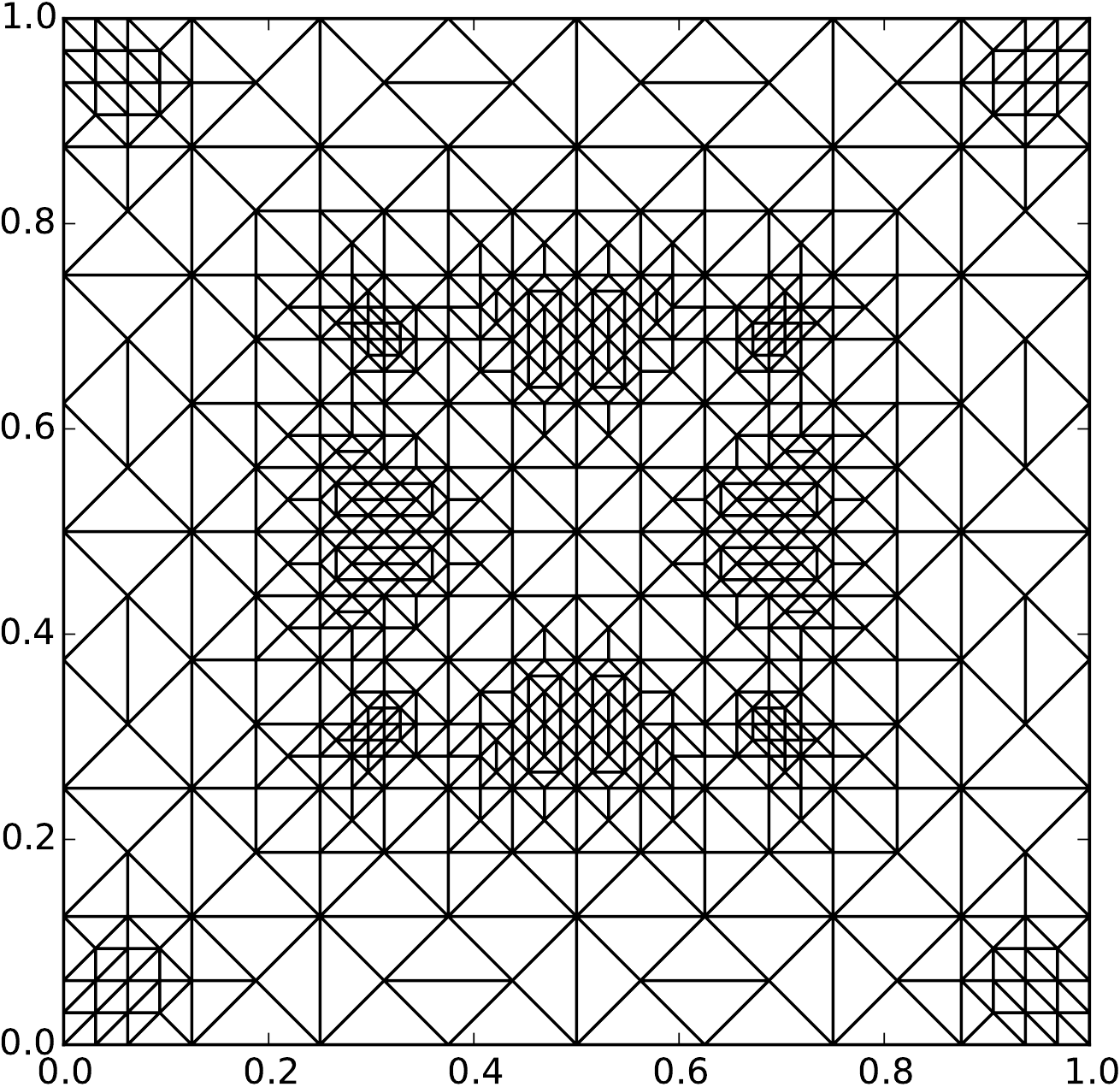}
    \caption{The sequence of adaptive meshes in  the elastic obstacle case with $\epsilon = 10^{-5}$.}
    \label{fig:meshe5}
\end{figure}
\begin{figure}
    \includegraphics[width=0.3\textwidth]{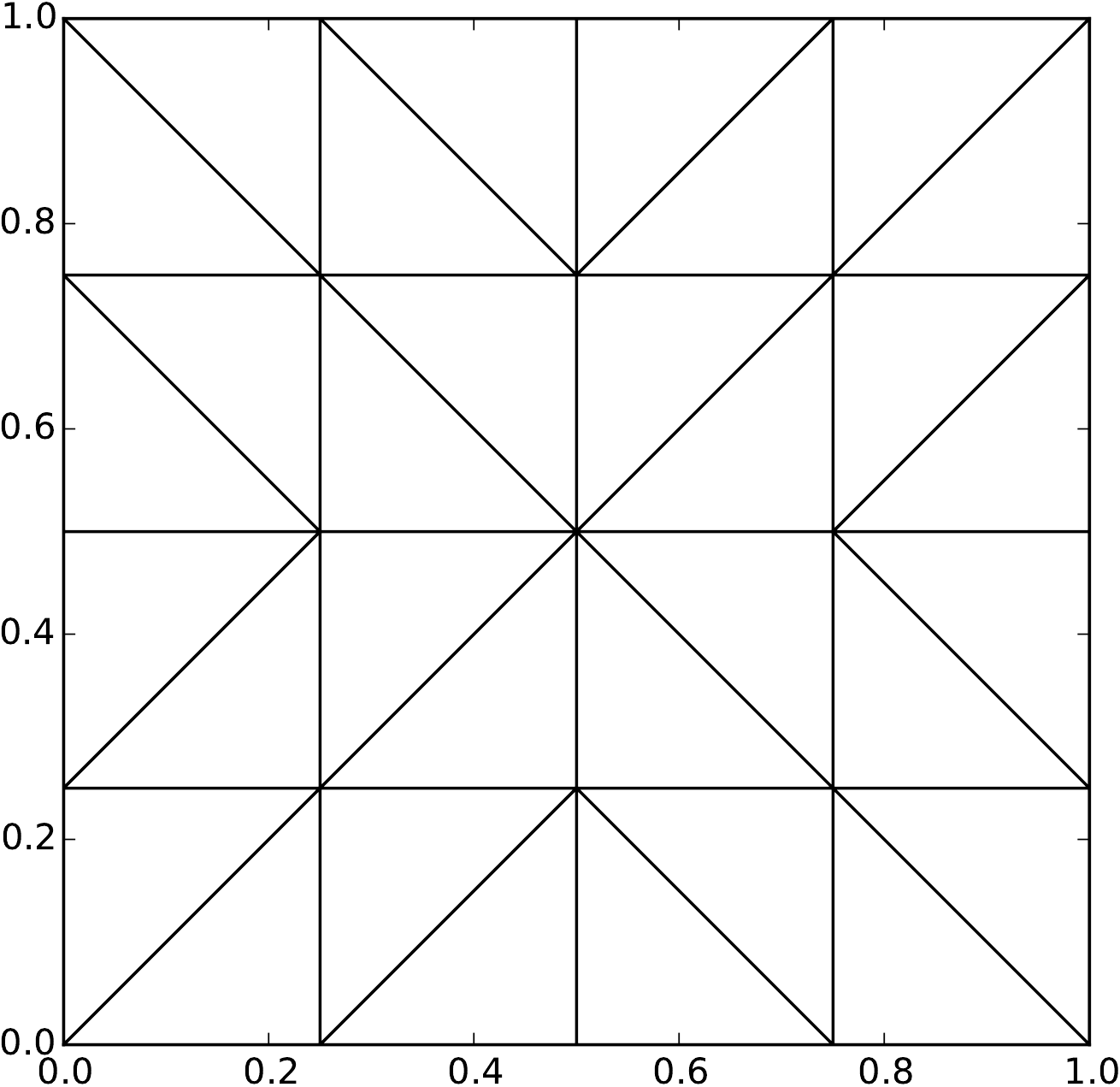}
    \includegraphics[width=0.3\textwidth]{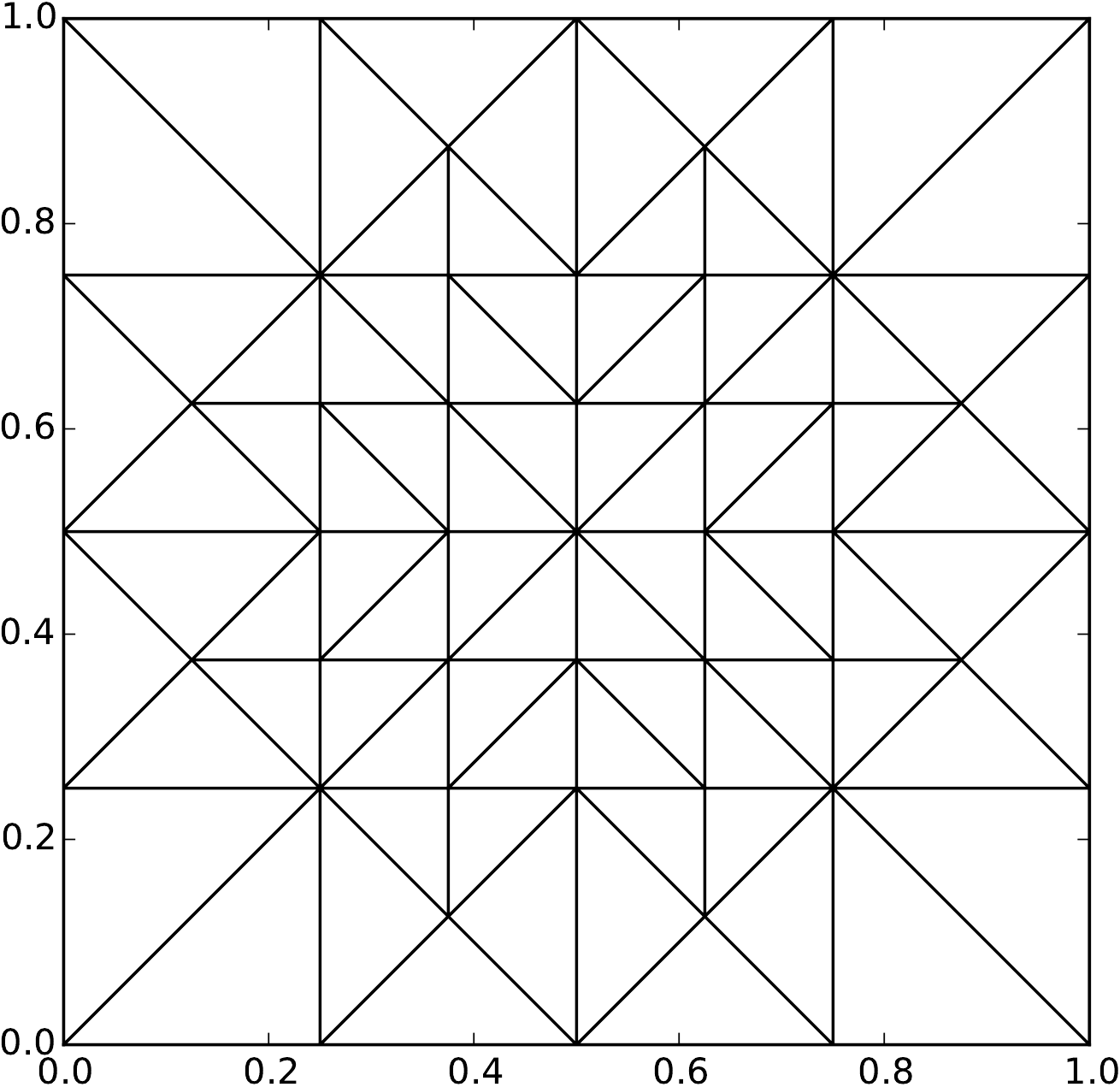}
    \includegraphics[width=0.3\textwidth]{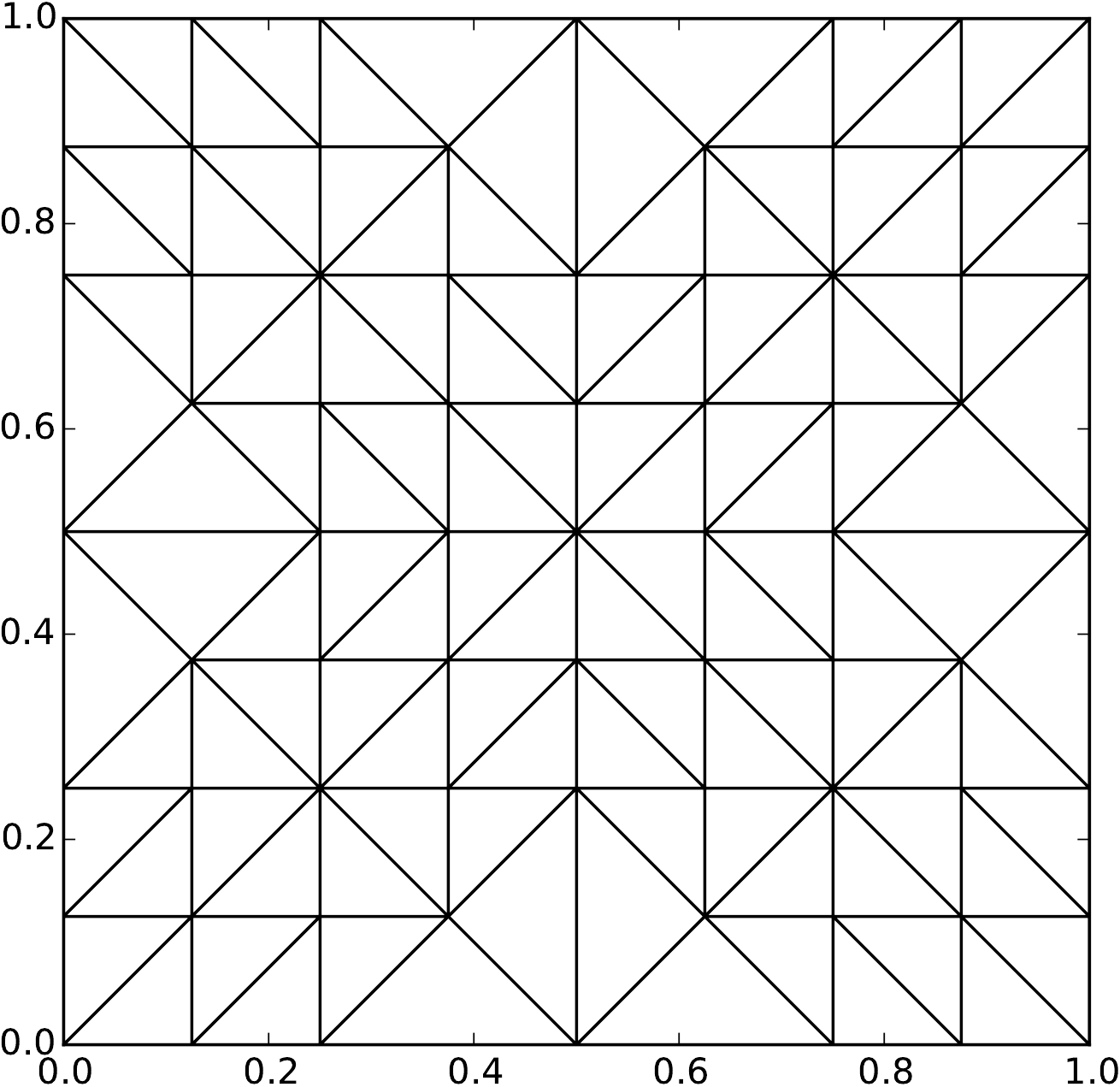}
    \includegraphics[width=0.3\textwidth]{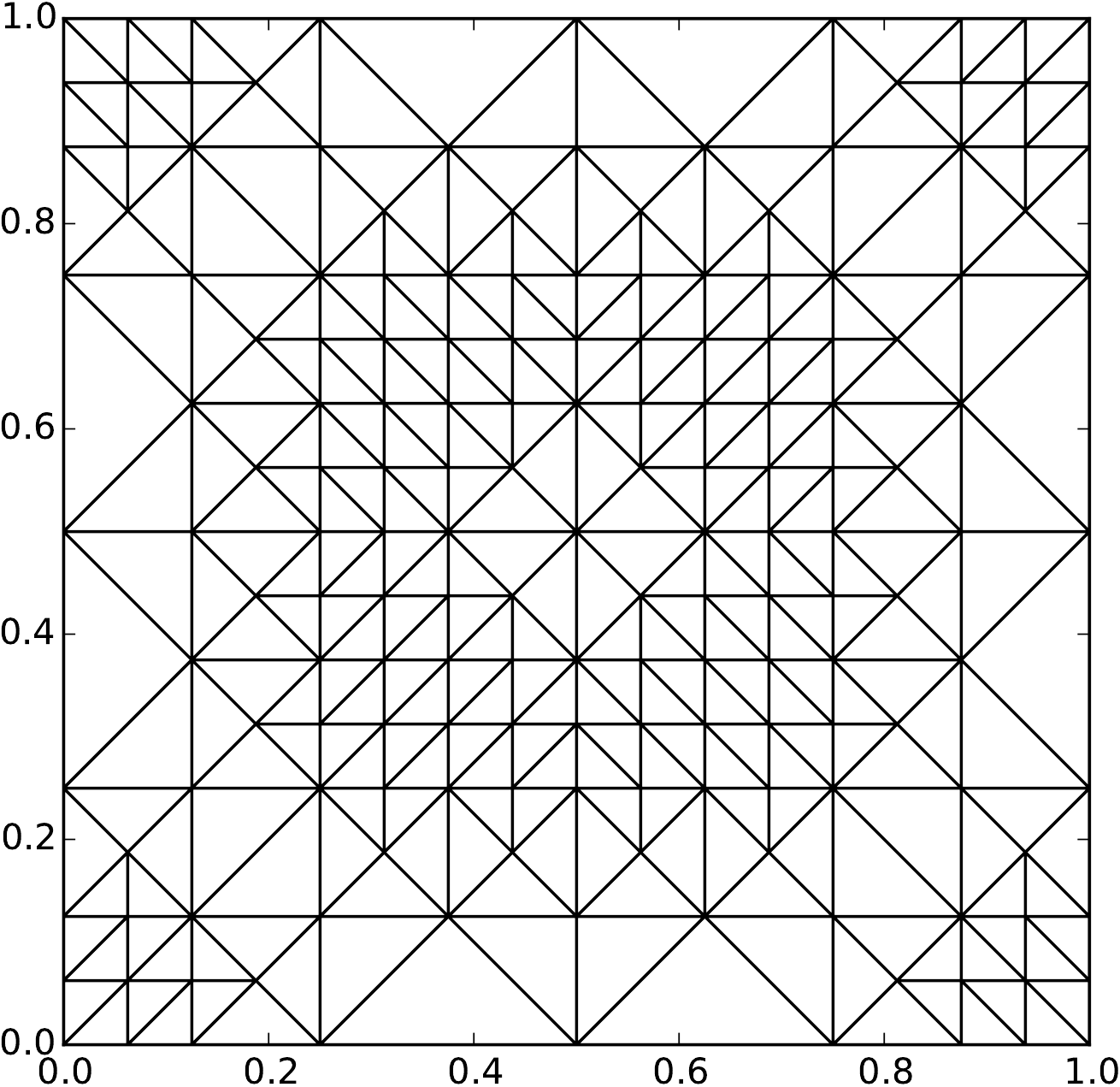}
            \quad  \ \ \ 
    \includegraphics[width=0.3\textwidth]{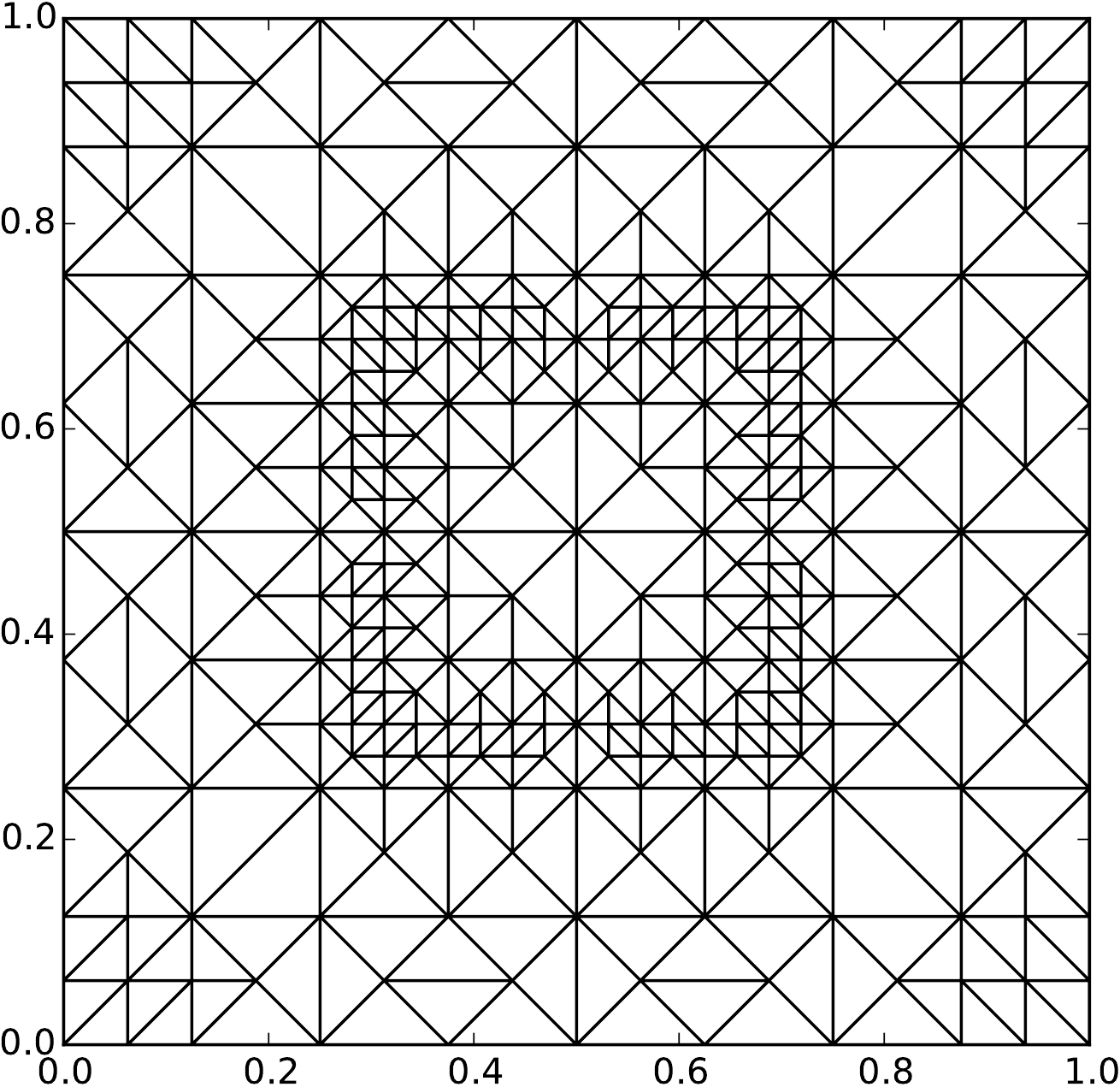}
            \quad  \ \ \ 
    \includegraphics[width=0.3\textwidth]{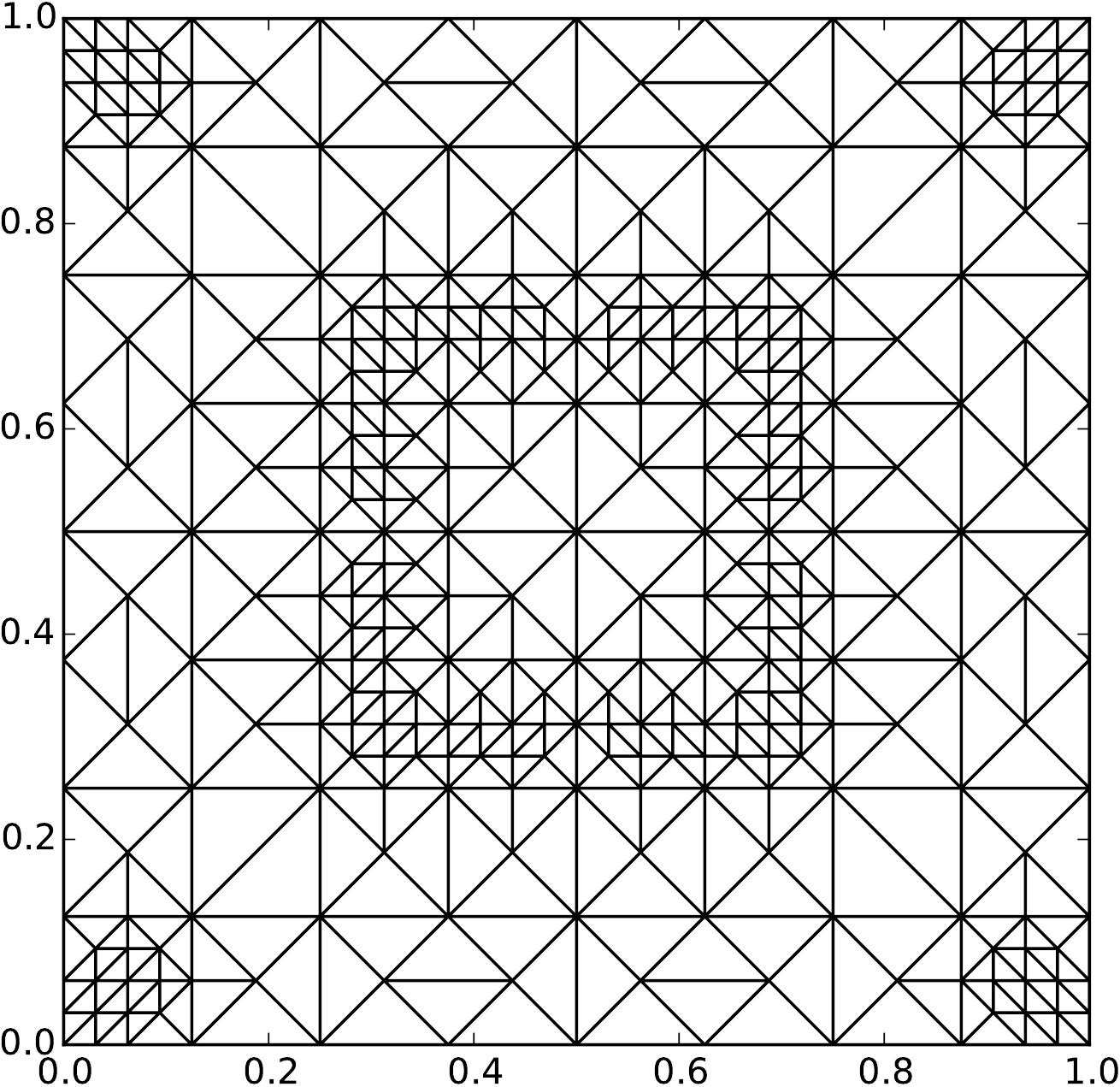}
    \caption{The sequence of adaptive meshes in  the elastic obstacle case with $\epsilon = 10^{-4}$.}
    \label{fig:meshe4}
\end{figure}
\begin{figure}
    \includegraphics[width=0.3\textwidth]{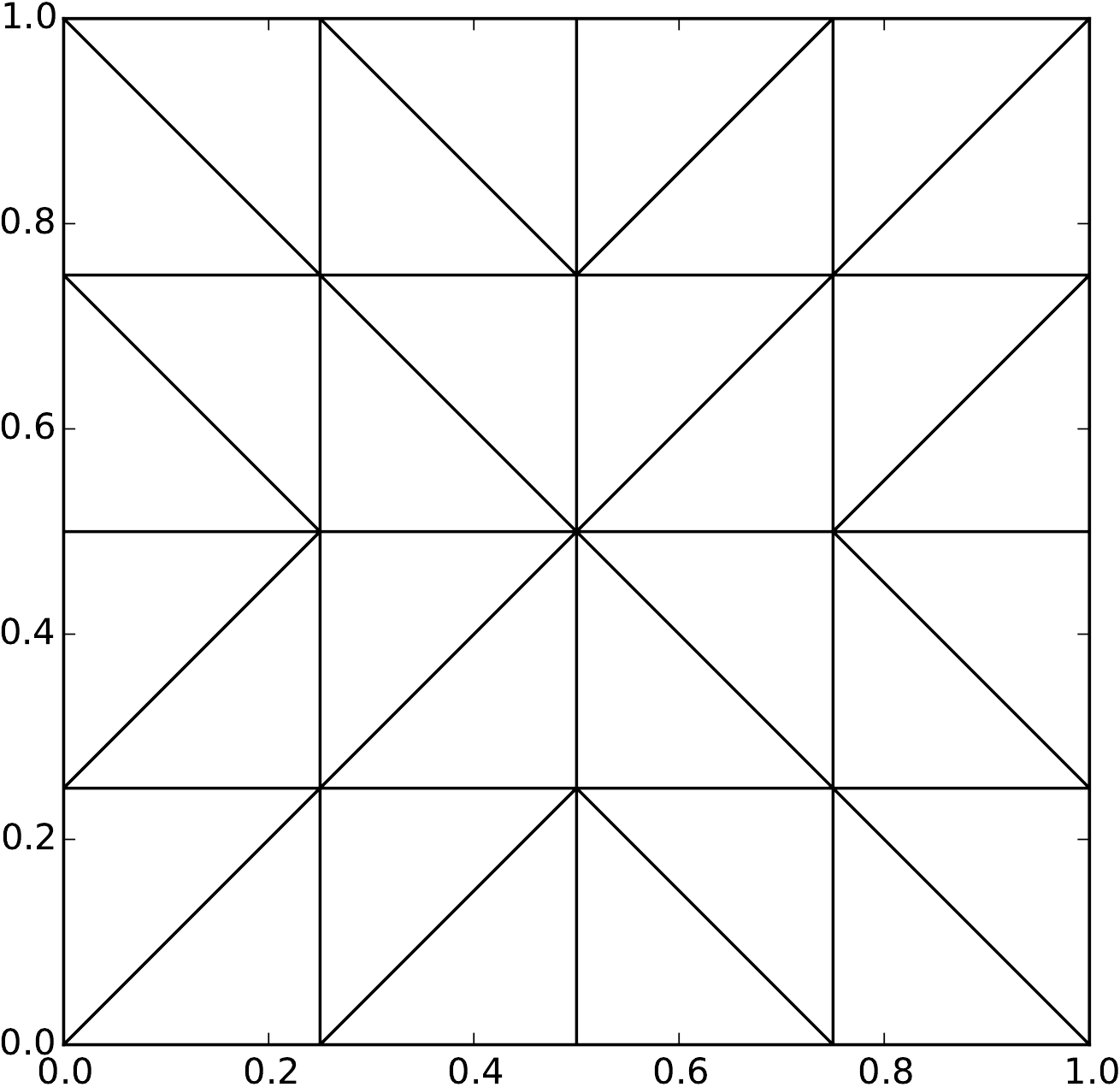}
    \includegraphics[width=0.3\textwidth]{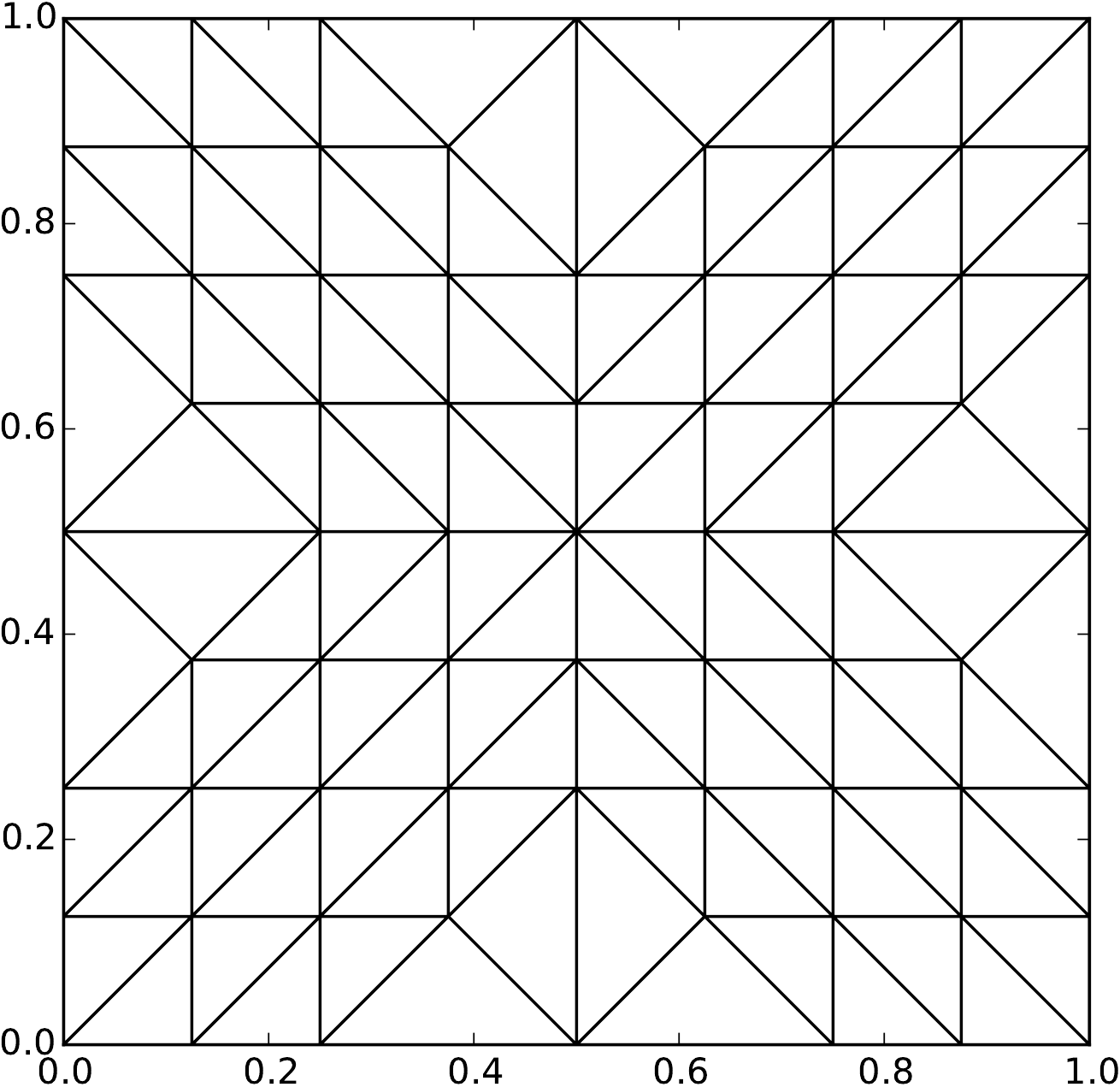}
    \includegraphics[width=0.3\textwidth]{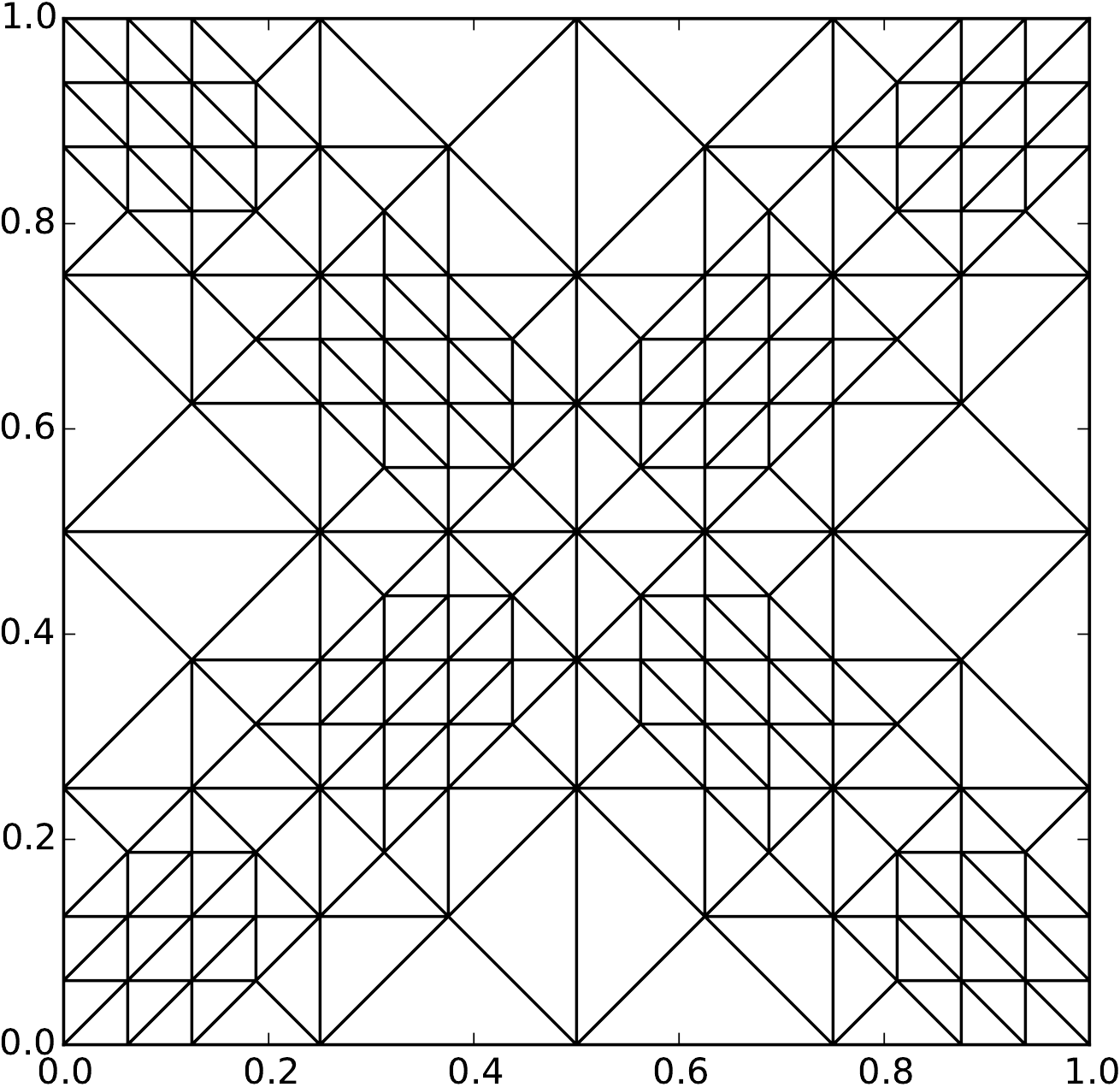}
    \includegraphics[width=0.3\textwidth]{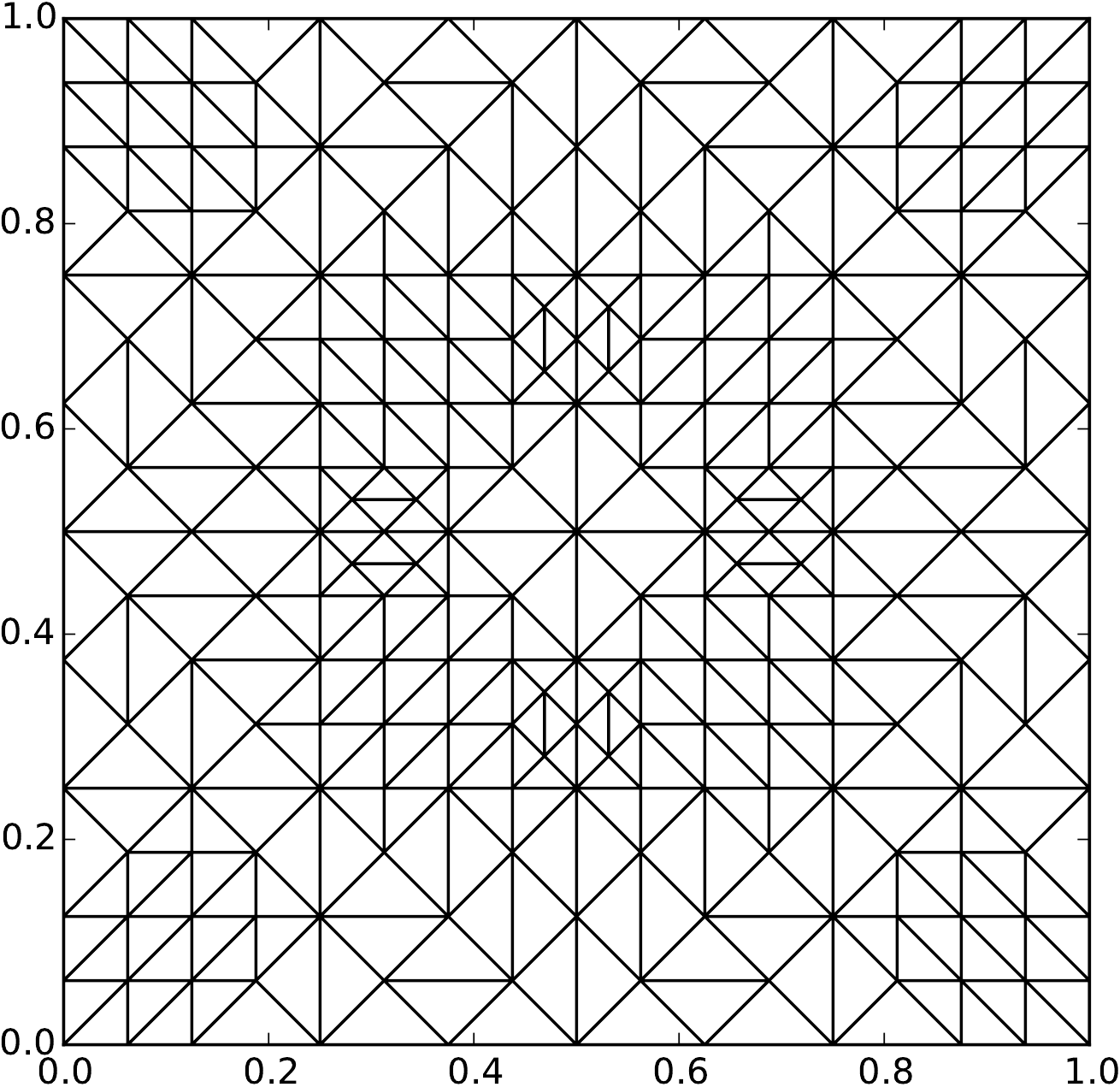}
            \quad  \ \ \ 
    \includegraphics[width=0.3\textwidth]{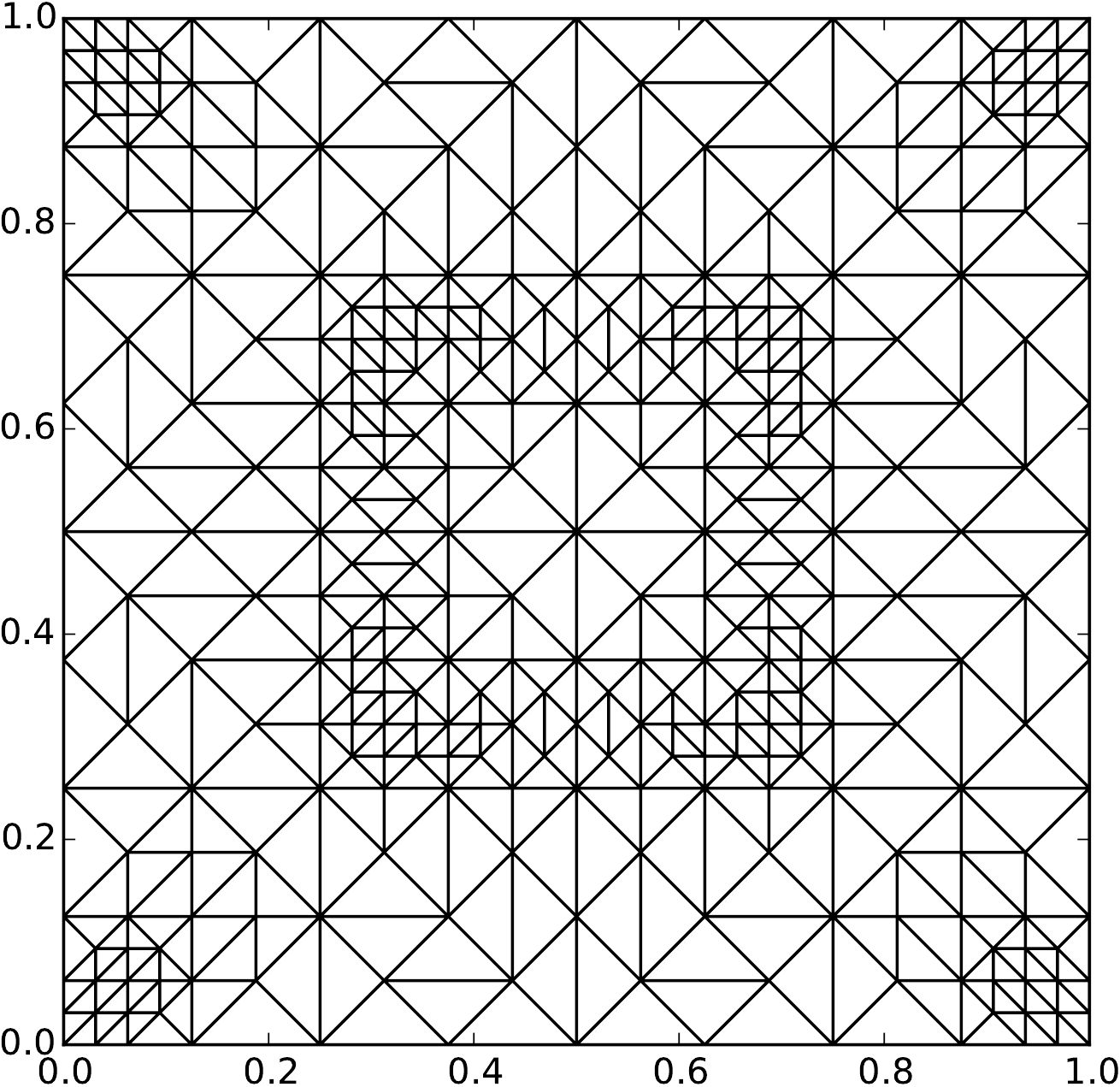}
            \quad  \ \ \ 
    \includegraphics[width=0.3\textwidth]{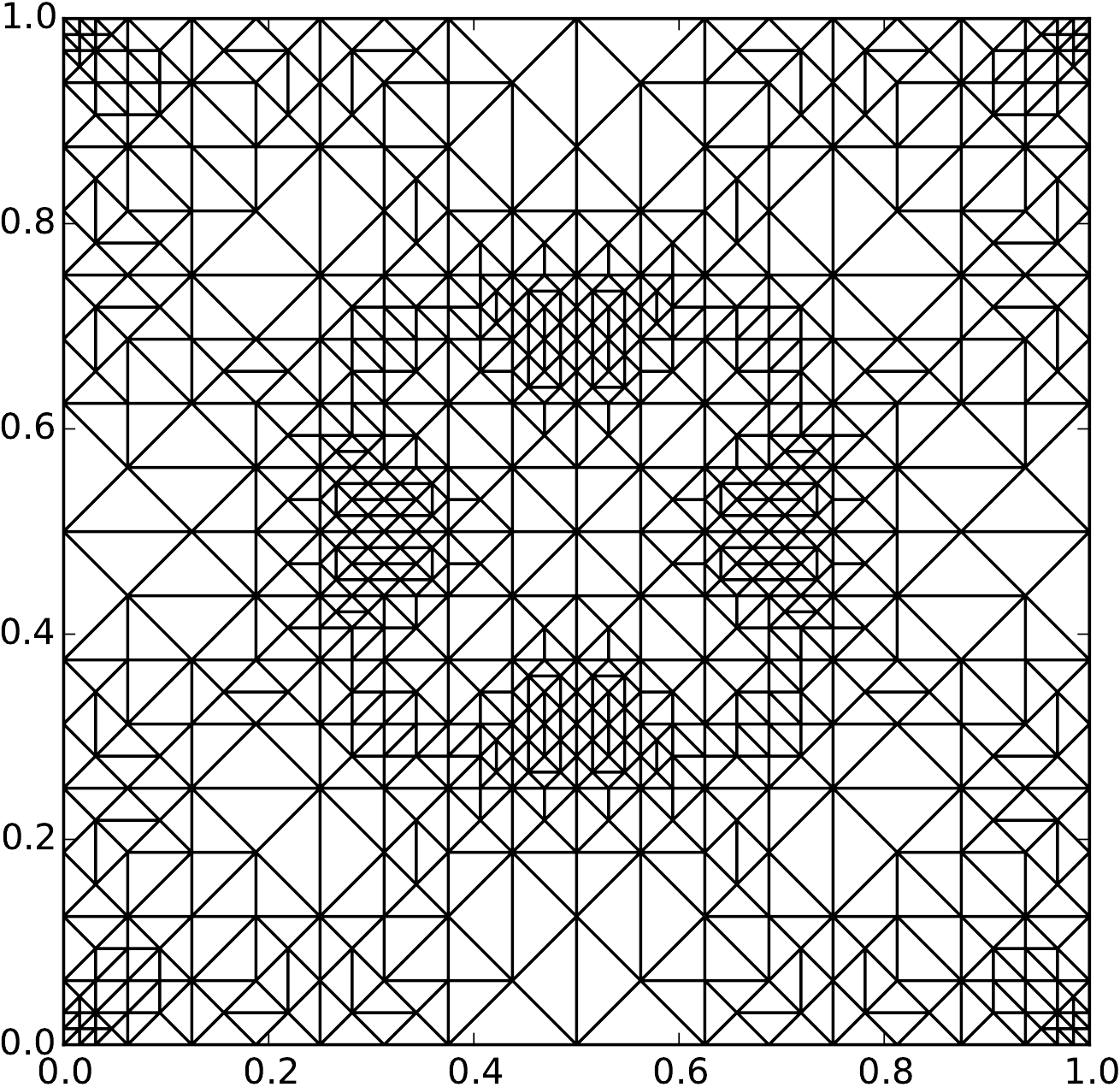}
    \caption{The sequence of adaptive meshes in  the elastic obstacle case with $\epsilon = 10^{-3}$.}
    \label{fig:meshe3}
\end{figure}


\bibliographystyle{spmpsci}      
\bibliography{PlateObst}

\end{document}